\documentclass[preprint,11pt]{elsarticle}

\usepackage{amsthm,amsmath}
\RequirePackage[colorlinks,citecolor=blue,urlcolor=blue,breaklinks=true]{hyperref}

\usepackage{amssymb,a4wide}
\usepackage{amsfonts,mathrsfs}
\usepackage{url}
\usepackage{pifont,dsfont}
\usepackage{breakcites,microtype}

\usepackage{natbib}
\setcitestyle{square, comma, numbers,sort&compress, super, year}

\usepackage{cleveref}
\usepackage{autonum}

\usepackage{nicefrac}
\usepackage{graphicx}

\usepackage{cleveref}
\usepackage{color}

\DeclareMathOperator{\tr}{tr}

\DeclareMathOperator*{\argmin}{arg\,min}

\def\ds{\displaystyle}

\def\tilde{\widetilde}

\newcommand{\RR}{\mathbb{R}}

\newcommand{\NN}{\mathbb{N}}

\newcommand{\eps}{\epsilon}

\newcommand{\balpha}{\boldsymbol{\alpha}}
\newcommand{\eeta}{\boldsymbol{\eta}}

\newcommand{\bzeta}{\boldsymbol{\zeta}}
\newcommand{\btheta}{\boldsymbol{\theta}}
\newcommand{\bvartheta}{\boldsymbol{\vartheta}}

\newcommand{\bDelta}{\boldsymbol{\Delta}}

\newcommand{\bxi}{\boldsymbol{\xi}}
\newcommand{\bfeta}{{\boldsymbol{\eta}}}

\newcommand{\bfE}{\mathbf E}
\newcommand{\bfH}{\mathbf H}

\newcommand{\bfM}{\mathbf M}

\newcommand{\bfI}{\mathbf I}

\newcommand{\bfSigma}{\boldsymbol\Sigma}

\newcommand{\bh}{\boldsymbol h}

\newcommand{\bu}{\boldsymbol u}
\newcommand{\bv}{\boldsymbol v}

\newcommand{\bU}{\boldsymbol U}
\newcommand{\bV}{\boldsymbol V}
\newcommand{\bW}{\boldsymbol W\!}

\newcommand{\bD}{\boldsymbol D}
\newcommand{\bL}{\boldsymbol L}

\newcommand{\bS}{\boldsymbol S}

\newcommand{\bX}{\boldsymbol X}
\newcommand{\bx}{\boldsymbol x}
\newcommand{\bY}{\boldsymbol Y}

\newcommand{\by}{\boldsymbol y}

\newtheorem{remark}{Remark}[section]

\newtheorem{lemma}{Lemma}

\newtheorem{proposition}{Proposition}
\newtheorem{theorem}{Theorem}

\usepackage[lined,longend,algoruled,linesnumbered,boxed]{algorithm2e} 
\DontPrintSemicolon
\usepackage{algorithmic}

\journal{Stochastic Processes and their Applications}

\begin{document}

\parindent=0pt

\begin{frontmatter}

\title{User-friendly guarantees for the Langevin Monte Carlo with inaccurate gradient}

\author{Arnak S. Dalalyan}
\author{Avetik Karagulyan}

\cortext[author] {Corresponding author.\\ \textit{E-mail address:} arnak.dalalyan@ensae.fr}
\address{CREST, ENSAE, 5 av.\ Henry Le Chatelier,
91120 Palaiseau, France.}

\begin{abstract}
In this paper, we study the problem of sampling from a given probability density
function that is known to be smooth and strongly log-concave. We analyze several
methods of approximate sampling based on discretizations of the (highly overdamped)
Langevin diffusion and establish guarantees on its error measured in the
Wasserstein-2 distance. Our guarantees improve or extend the state-of-the-art
results in three directions. First, we provide an upper bound on the error of the
first-order Langevin Monte Carlo (LMC) algorithm with optimized varying step-size.
This result has the advantage of being horizon free (we do not need to know in
advance the target precision) and to improve by a logarithmic factor the
corresponding result for the constant step-size. Second, we study the case where
accurate evaluations of the gradient of the log-density are unavailable, but one
can have access to approximations of the aforementioned gradient. In such a situation,
we consider both deterministic and stochastic approximations of the gradient and
provide an upper bound on the sampling error of the first-order LMC that quantifies
the impact of the gradient evaluation inaccuracies. Third, we establish upper bounds
for two versions of the second-order LMC, which leverage the Hessian of the
log-density. We provide nonasymptotic guarantees on the sampling error of these second-order
LMCs. These guarantees reveal that the second-order LMC algorithms improve on the
first-order LMC in ill-conditioned settings.
\end{abstract}


\end{frontmatter}


\section{Introduction}

The problem of sampling a random vector distribu\-ted according to a given
target distribution is central in many applications. In the present paper,
we consider this problem in the case of a target distribution having a smooth
and log-concave density $\pi$ and when the sampling is performed by a version of
the Langevin Monte Carlo algorithm (LMC). More precisely,
for a positive integer $p$, we consider a continuously differentiable
function $f:\RR^p\to\RR$ satisfying the following assumption:
For some positive constants $m$ and $M$, it holds
\begin{equation} \label{1}
\begin{cases}
f(\btheta)-f(\btheta')-\nabla f(\btheta')^\top (\btheta-\btheta')
\ge (\nicefrac{m}2)\|\btheta-\btheta'\|_2^2, \text{\vphantom{$I_{\textstyle\int_{I_I}}$}}\\
\|\nabla f(\btheta)-\nabla f(\btheta')\|_2 \le M \|\btheta-\btheta'\|_2,
\end{cases}
\qquad \forall \btheta,\btheta'\in\RR^p,
\end{equation}
where $\nabla f$ stands for the gradient of $f$ and $\|\cdot\|_2$ is the
Euclidean norm. The target distributions considered in this paper are those
having a density with respect to the Lebesgue measure on $\RR^p$ given by
\begin{equation}
\pi(\btheta) = \frac{e^{-f(\btheta)}}{\int_{\RR^p} e^{-f(\bu)}\,d\bu}.
\end{equation}
We say that the density $\pi(\btheta)\propto e^{-f(\btheta)}$ is log-concave
(resp.\ strongly log-concave) if the function $f$ satisfies the first inequality
of (\ref{1}) with $m=0$ (resp.\ $m>0$).

Most part of this work focused on the analysis of the LMC algorithm, which can
be seen as the analogue in the problem of sampling of the gradient descent
algorithm for optimization. For a sequence of positive parameters
$\bh=\{h_k\}_{k\in\NN}$, referred to as the step-sizes and for an initial point
$\bvartheta_{0,\bh}\in\RR^p$  that may be deterministic or random, the iterations of
the LMC algorithm are defined by the update rule
\begin{align}\label{2}
\bvartheta_{k+1,\bh} = \bvartheta_{k,\bh} - h_{k+1} \nabla f(\bvartheta_{k,\bh})+ \sqrt{2h_{k+1}}\;\bxi_{k+1};
\qquad k=0,1,2,\ldots
\end{align}
where  $\bxi_1,\ldots,\bxi_{k},\ldots$ is a sequence of mutually independent, and independent of $\bvartheta_{0,\bh}$, centered Gaussian vectors with covariance
matrices equal to identity.

When all the $h_k$'s are equal to some value $h>0$, we will call the sequence
in \eqref{2} the constant step LMC and will denote it by $\bvartheta_{k+1,h}$.
When $f$ satisfies assumptions \eqref{1}, if $h$ is small and $k$ is large (so that
the product $kh$ is large), the distribution of $\bvartheta_{k,h}$ is known
to be a good approximation to the distribution with density $\pi(\btheta)$. An
important question is to quantify the quality of this approximation. An appealing approach to address this question is by establishing non asymptotic upper bounds on
the error of sampling; this kind of bounds are particularly useful for deriving a stopping rule for the LMC algorithm, as well as for understanding the computational
complexity of sampling methods in high dimensional problems. In the present paper
we establish such bounds by focusing on their user-friendliness. The latter means
that our bounds are easy to interpret, hold under conditions that are not difficult
to check and lead to simple theoretically grounded choice of the number of
iterations and the step-size.

In the present work, we measure the error of sampling in the
Wasserstein-Monge-Kantorovich distance $W_2$. For two measures $\mu$ and $\nu$
defined on $(\RR^p,\mathscr B(\RR^p))$, and
for a real number $q\ge 1$, $W_q$ is defined by
\begin{equation}
W_q(\mu,\nu) = \Big(\inf_{\varrho\in \varrho(\mu,\nu)} \int_{\RR^p\times \RR^p}
		\|\btheta-\btheta'\|_2^q\,d\varrho(\btheta,\btheta')\Big)^{1/q},
\end{equation}
where the $\inf$ is with respect to all joint distributions $\varrho$ having $\mu$ and
$\nu$ as marginal distributions. For statistical and machine learning applications,
we believe that this distance is more suitable for assessing the quality of approximate
sampling schemes than other metrics such as the total variation or the Kullback-Leibler
divergence. Indeed, bounds on the Wasserstein distance---unlike the
bounds on the total-variation---provide direct guarantees on the accuracy of approximating the first and the second order moments.

Asymptotic properties of the LMC algorithm, also known as Unadjusted Langevin Algorithm (ULA),
and its Metropolis adjusted version, MALA, have been studied in a number of papers
\citep{RobertsTweedie96,RobertsRosenthal98,StramerTweedie99-1,StramerTweedie99-2,Jarner2000,
RobertsStramer02}. These results do not emphasize the effect of the dimension on the computational
complexity of the algorithm, which is roughly proportional to the number of iterations.
Non asymptotic bounds on the total variation error of the LMC for log-concave and strongly
log-concave distributions have been established by \cite{Dalalyan14}. If a  warm start is
available, the results in \cite{Dalalyan14} imply that after $O(p/\eps^2)$ iterations the
LMC algorithm has an error bounded from above by $\eps$. Furthermore, if we assume that
in addition to \eqref{1} the function $f$ has a Lipschitz continuous Hessian, then a modified
version of the LMC, the LMC with Ozaki discretization (LMCO), needs $O(p/\eps)$ iterations
to achieve a precision level $\eps$. These results were improved and extended to the Wasserstein
distance by \citep{durmus2017,Durmus2}. More precisely, they removed the condition of the warm
start and proved that under the Lipschitz continuity assumption on the Hessian of $f$, it is
not necessary to modify the LMC for getting the rate $O(p/\eps)$. The last result is closely
related to an error bound between a diffusion process and its Euler discretization established
by \cite{alfonsi2014}.

On a related note, \citep{Bubeck18} studied the convergence of the LMC algorithm with reflection
at the boundary of a compact set, which makes it possible to sample from a compactly supported
density (see also \citep{brosse17a}). Extensions to non-smooth densities were presented in
\citep{Durmus3,Luu}. \citep{Cheng1} obtained guarantees similar to those in \citep{Dalalyan14}
when the error is measured by the Kullback-Leibler divergence. Very recently, \citep{Cheng2}
derived non asymptotic guarantees for the kinetic LMC which turned out to improve on the
previously known results. Langevin dynamics was used in \citep{Ridgway16,Brosse17b} in order to
approximate normalizing constants of target distributions. \cite{huggins17a} established tight
bounds in Wasserstein distance between the invariant distributions of two (Langevin) diffusions;
the bounds involve mixing rates of the diffusions and the deviation in their drifts.

The goal of the present work is to push further the study of the LMC and its variants both by
improving the existing guarantees and by extending them in some directions. Our main contributions
can be summarized as follows:
\begin{itemize}

\item We state simplified guarantees in Wasserstein distance with improved constants
both for the LMC and the LMCO when the step-size is constant, see \Cref{thOne} and
\Cref{thFive}.

\item We propose a varying-step LMC which avoids a logarithmic factor in the number
of iterations required to achieve a precision level $\eps$, see \Cref{thOneBis}.

\item We extend the previous guarantees to the case where accurate evaluations of the gradient are
unavailable. Thus, at each iteration of the algorithm, the gradient is computed within an error that
has a deterministic and a stochastic component. \Cref{thTwo} deals with functions $f$ satisfying \eqref{1},
whereas \Cref{thFour} requires the additional assumption of the smoothness of the Hessian of $f$.

\item We propose a new second-order sampling algorithm termed LMCO'. It has a per-iteration computational
cost comparable to that of the LMC and enjoys nearly the same guarantees as the LMCO, when the Hessian
of $f$ is Lipschitz continuous, see \Cref{thFive}.

\item We provide a detailed discussion of the relations between, on the one hand, the sampling methods
and guarantees of their convergence and, on the other hand, optimization methods and guarantees of their
convergence (see \Cref{secOpt}).
\end{itemize}
We have to emphasize right away that \Cref{thOne} is a corrected version of
\citep[Theorem 1]{DalalyanColt}, whereas \Cref{thTwo} extends \citep[Theorem 3]{DalalyanColt}
to more general noise. In particular, \Cref{thTwo} removes the unbiasedness and independence
conditions.  Furthermore, thanks to a shrewd use of a recursive inequality, the upper bound
in~\Cref{thTwo} is tighter than the one in~\citep[Theorem 3]{DalalyanColt}.

As an illustration of the first two bullets mentioned in the above summary of our contributions,
let us consider the following example. Assume that $m=10$, $M=20$
and we have at our disposal an initial sampling distribution $\nu_0$ satisfying
$W_2(\nu_0,\pi) = p + (p/m)$. The main inequalities in \Cref{thOne}  and
\Cref{thOneBis} imply that after $K$ iterations, the distribution $\nu_K$
obtained by the LMC algorithm satisfies
\begin{align}\label{G1}
W_2(\nu_K, \pi) \le  (1-mh)^K W_2(\nu_0,\pi) + 1.65(M/m)(hp)^{1/2}
\end{align}
for the constant step LMC and
\begin{align}\label{G2}
W_2(\nu_K,\pi)\le \frac{3.5M\sqrt{p}}{m\sqrt{M+m+(\nicefrac23)m(K-K_1)}}
\end{align}
for the varying-step LMC, where $K_1$ is an integer the precise value of which
is provided in \Cref{thOneBis}. One can compare these inequalities with the corresponding
bound in \citep{Durmus2}: adapted to the constant-step, it takes the form
\begin{align}
W_2^2(\nu_{K}, \pi) \le &
2\Big(1-\frac{mMh}{m+M}\Big)^K W^2_2(\nu_0,\pi)\\
&\ +
\frac{Mhp}{m}(m+M)\Big(h + \frac{m+M}{2mM}\Big)\Big(2+\frac{M^2h}{m}+\frac{M^2h^2}{6}\Big).
\label{G3}
\end{align}
For any $\eps>0$, we can derive from these guarantees the smallest number of iterations,
$K_\eps$, for which there is a $h>0$ such that the corresponding upper bound is smaller
than $\eps$. The logarithms of these values $K_\eps$ for varying $\eps\in\{0.001,0.005,0.02\}$
and $p\in\{25,\ldots,1000\}$ are plotted in \Cref{fig1}. We observe that for all the considered values
of $\eps$ and $p$, the number of iterations derived from \eqref{G2} (referred to as \Cref{thOneBis})
is smaller than those derived from \eqref{G1} (referred to as \Cref{thOne}) and from
\eqref{G3} (referred to as DM bound). The difference between the varying-step LMC and the constant
step LMC becomes more important when the target precision level $\eps$ gets smaller.
In average over all values of $p$, when $\eps = 0.001$, the number of iterations derived from
\eqref{G3} is 4.6 times larger than that derived from \eqref{G2}, and almost $3$ times larger
than the number of iterations derived from \eqref{G1}.

\begin{figure}
\includegraphics[width=0.8\textwidth]{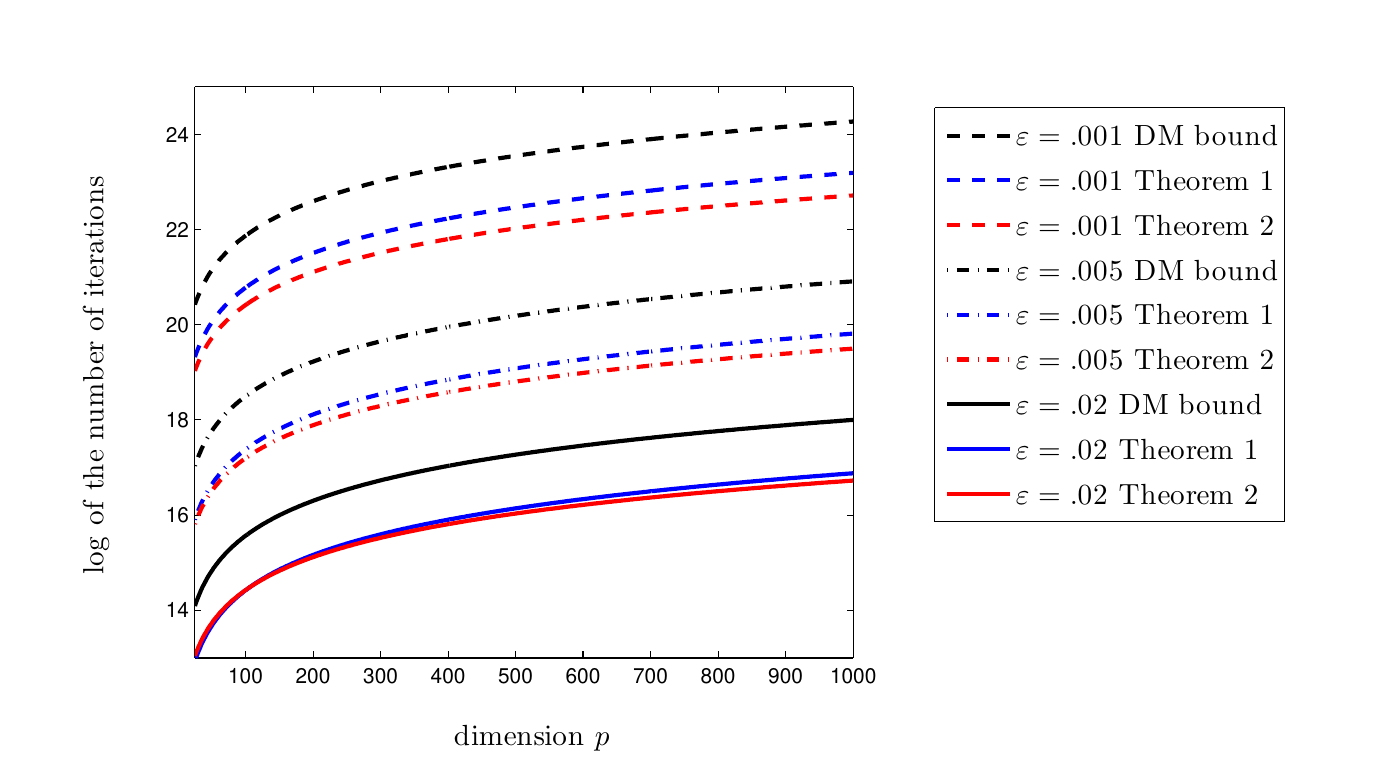}
\caption{Plots showing the logarithm of the number of iterations as function of dimension $p$ for
several values of $\eps$. The plotted values are derived from \eqref{G1}-\eqref{G3} using the data $m=10$,
$M=20$, $W_2^2(\nu_0,\pi)=p + (p/m)$.}
\label{fig1}
\end{figure}

\section{Guarantees in the Wasserstein distance with accurate gradient}
\label{sec:1}

The rationale behind the LMC \eqref{2} is simple: the Markov chain
$\{\bvartheta_{k,\bh}\}_{k\in\NN}$ is the Euler discretization of a continuous-time
diffusion process $\{\bL_t :t\in\RR_+\}$, known as Langevin diffusion. The latter is
defined by the stochastic differential equation
\begin{equation}\label{3}
d\bL_t = -\nabla f(\bL_t)\,dt + \sqrt{2} \; d\bW_t,\qquad t\ge 0,
\end{equation}
where $\{\bW_t:t\ge 0\}$ is a $p$-dimensional Brownian motion. When $f$ satisfies
condition (\ref{1}), equation (\ref{3}) has a unique strong solution, which is a Markov
process. Furthermore, the process $\bL$ has $\pi$ as invariant
density~\citep[Thm. 3.5]{bhattacharya1978}. Let $\nu_k$ be the distribution of the
$k$-th iterate of the LMC algorithm, that is $\vartheta_{k,\bh}\sim \nu_k$. In what
follows, we present user-friendly guarantees on the closeness of $\nu_k$ and $\pi$,
when $f$ is strongly convex.

\subsection{Reminder on guarantees for the constant-step LMC}

When the function $f$ is $m$-strongly convex and $M$-gradient Lipschitz,
upper bounds on the sampling error measured in Wasserstein distance of the
LMC algorithm have been established in \citep{Durmus2,DalalyanColt}.
We state below a slightly adapted version of their result, which will serve
as a benchmark for the bounds obtained in this work.

\begin{theorem}\label{thOne}
Assume that $h\in(0,\nicefrac2M)$ and $f$ satisfies condition \eqref{1}.
The following claims hold:
\begin{enumerate}\itemsep=6pt
\item[{\rm (a)}\!\!] \!\! If $h\le \nicefrac2{(m+M)}$ then $W_2(\nu_K, \pi) \le
(1-mh)^K W_2(\nu_0,\pi) + 1.65(\frac{M}{m})(hp)^{1/2}$\!\!.
\item[{\rm (b)}\!\!] \!\! If $h\ge \nicefrac2{(m+M)}$ then $W_2(\nu_K, \pi) \le
\ds (Mh-1)^K W_2(\nu_0,\pi) + \frac{1.65 Mh}{2-Mh}(hp)^{1/2}$\!\!.
\end{enumerate}
\end{theorem}
We refer the readers interested in the proof of this theorem
either to \citep{DalalyanColt} or to \Cref{secProof}, where the latter is obtained
as a direct consequence of \Cref{thTwo}. The factor $1.65$ is obtained by
upper bounding $7\sqrt{2}/6$.

In practice, a relevant approach to getting an accuracy of at most $\epsilon$ is
to  minimize the upper bound provided by \Cref{thOne} with respect to $h$, for a fixed $K$.
Then, one can choose the smallest $K$ for which the obtained upper bound is smaller than
$\epsilon$. One useful observation is that the upper bound of case (b) is an increasing
function of $h$. Its minimum is always attained at $h= 2/(m+M)$, which means that
one can always look for a step-size in the interval $(0,2/(m+M)]$ by minimizing
the upper bound in (a).  This can be done using standard line-search methods such as the
bisection algorithm.

Note that if the initial value $\bvartheta_{0}=\btheta_{0}$ is deterministic then, using the notation
$\btheta^* = \argmin_{\theta\in\RR^p} f(\btheta)$, in view of \citep[Proposition 1]{Durmus2}, we have
\begin{align}
W_2(\nu_0,\pi)^2
	& = \int_{\RR^p} \|\btheta_{0}-\btheta\|_2^2\pi(d\btheta)
	 \le \|\btheta_{0}-\btheta^*\|_2^2 + p/m.\label{4}
\end{align}
Finally, let us remark that if we choose $h$ and $K$ so that
\begin{equation}\label{5}
h\le \nicefrac{2}{(m+M)},\qquad e^{-mhK}W_2(\nu_0,\pi)\le \varepsilon/2,\quad
1.65(M/m)(hp)^{1/2}\le \varepsilon/2,
\end{equation}
then we have $W_2(\nu_K, \pi) \le \varepsilon$. In other words, conditions
\eqref{5} are sufficient for the density of the output of the LMC algorithm after $K$
iterations to be within the precision $\varepsilon$ of the target density when the precision
is measured using the Wasserstein distance. This readily yields
\begin{equation}\label{6}
h\le \frac{m^2\varepsilon^2}{11M^2p}\wedge \frac2{m+M}\quad\text{and}\quad
hK\ge \frac1m\log\Big(\frac{2(\|\btheta_{0}-\btheta^*\|_2^2+ p/m)^{1/2}}\varepsilon\Big)
\end{equation}
Assuming $m,M$ and $\|\btheta_{0}-\btheta^*\|_2^2/p$ to be constants, we can deduce
from the last display that it suffices $K = C (p/\varepsilon^{2})\log(p/\varepsilon^2)$
number of iterations in order to reach the precision level $\varepsilon$. This fact has
been first established in \citep{Dalalyan14} for the LMC algorithm with a warm start
and the total-variation distance. It was later improved by \cite{durmus2017,Durmus2},
where the authors showed that the same result holds for any starting point and established similar
bounds for the Wasserstein distance. \Cref{thOne} above can be seen as a user-friendly version
of the corresponding result established by \cite{Durmus2}.

\begin{remark}
Although~\eqref{4} is relevant for understanding the order of magnitude of
$W_2(\nu_0,\pi)$, it has limited applicability since the distance
$\|\btheta_0-\btheta^*\|$ might be hard to evaluate. As mentioned in \citep{DalalyanColt},  an attractive
alternative to that bound is given by the inequality \footnote{The second line follows from
strong convexity whereas the third line is a consequence of the fact that $\btheta^*$
is a stationary point of $f$.}
\begin{align}
m W_2(\nu_0,\pi)^2
	& \le  m\|\btheta_{0}-\btheta^*\|_2^2 + p\\
	& \le 2\big(f(\btheta_0)-f(\btheta^*)-\nabla f(\btheta^*)^\top(\btheta_0-\btheta)
	\big)+ p\\
	& = 2\big(f(\btheta_0)-f(\btheta^*)\big)+p.\label{init1}
\end{align}
If $f$ is lower bounded by some known constant, for instance if $f\ge 0$,
the last inequality provides the computable upper bound $W_2(\nu_0,\pi)^2 \le
\big(2f(\btheta_0)+p\big)/m$.
\end{remark}

\subsection{Guarantees under strong convexity for the varying step LMC}

The result of previous section provides a guarantee for the constant step LMC. One may wonder
if using a variable step sizes $\bh = \{h_k\}_{k\in\NN}$ can improve the convergence. Note that
in \citep[Theorem 5]{Durmus2}, guarantees for the variable step LMC are established. However,
they do not lead to a clear message on the choice of the step-sizes. The next result fills this gap
by showing that an appropriate selection of step-sizes improves on the constant step LMC with  an
improvement factor logarithmic in $p/\epsilon^2$.

\begin{theorem}\label{thOneBis}
Let us consider the LMC algorithm with varying step-size $h_{k+1}$ defined
by
\begin{equation}\label{h2}
h_{k+1}  = \frac{2}{M+m+(\nicefrac23)m(k-K_1)_+},\qquad k=1,2,\ldots
\end{equation}
where $K_1$ is the smallest non-negative integer
satisfying\footnote{Combining the definition of $K_1$ and the upper bound in \eqref{4},
one easily checks that if $\|\btheta_0-\btheta^*\|_\infty$ is bounded, then $K_1$ is
upper bounded by a constant that does not depend on the dimension $p$.}
\begin{equation}\label{K1}
K_1\ge \frac{\ln \big(W_2(\nu_0,\pi)/\sqrt{p}\big) + \ln(m/M) + (\nicefrac12)\ln(M+m)}{\ln(1+\nicefrac{2m}{M-m})}.
\end{equation}
If $f$ satisfies \eqref{1}, then for every $k\ge K_1$, we have
\begin{equation}\label{W2}
W_2(\nu_k,\pi)\le \frac{3.5M\sqrt{p}}{m\sqrt{M+m+(\nicefrac23)m(k-K_1)}}.
\end{equation}
\end{theorem}

The step size  \eqref{h2} has two important advantages as compared to the constant steps.
The first advantage is that it is independent of the target precision level $\epsilon$. The second
advantage is that we get rid of the logarithmic terms in the number of iterations required to achieve
the precision level $\epsilon$. Indeed, it suffices $K = K_1 + (27M^2/2m^3)(p/\eps^2)$ iterations to
get the right hand side of \eqref{W2} smaller than $\eps$, where $K_1$ depends neither on the
dimension $p$ nor on the precision level $\eps$.

Since the choice of $h_{k+1}$ in \eqref{h2} might appear mysterious, we provide below a quick explanation of the
main computations underpinning this choice.
The main step of the proof of upper bounds on $W_2(\nu_k,\pi)$, is the following recursive
inequality (see \Cref{propA} in \Cref{secProof})
\begin{align}
W_2(\nu_{k+1},\pi) \le (1-mh_{k+1}) W_2(\nu_k,\pi) + 1.65 M\sqrt{p}\, h_{k+1}^{3/2}.
\end{align}
Using the notation $B_k = \frac{2(m/3)^{3/2}}{1.65M\sqrt{p}}W_2(\nu_k,\pi)$, this inequality
can be rewritten as
$$
B_{k+1}\le (1-mh_{k+1}) B_k + 2(mh_{k+1}/3)^{3/2}.
$$
Minimizing the right hand side with respect to $h_{k+1}$, we find that the minimum is attained at
the stationary point
\begin{align}\label{H2}
h_{k+1} = \frac{3}{m} B_k^2.
\end{align}
With this $h_{k+1}$, one checks that the sequence $B_k$ satisfies the recursive inequality
$$
B_{k+1}^2\le B_k^2(1-B_k^2)^2\le \frac{B_k^2}{1+B_k^2}.
$$
The function $g(x) = x/(1+x)$ being increasing in $(0,\infty)$, we get
$$
B_{k+1}^2\le \frac{B_k^2}{1+B_k^2}
		\le \frac{\frac{B_{k-1}^2}{1+B_{k-1}^2}}{1+\frac{B_{k-1}^2}{1+B_{k-1}^2}}
		= \frac{B_{k-1}^2}{1+2B_{k-1}^2}.
$$
By repetitive application of the same argument, we get
$$
B_{k+1}^2\le
		 \frac{B_{K_1}^2}{1+(k+1-K_1)B_{K_1}^2}.
$$
The integer $K_1$ was chosen so that $B_{K_1}^2\le \frac{2m}{3(M+m)}$, see \eqref{F1}.
Inserting this upper bound in the right hand side of the last display, we get
$$
B_{k+1}^2\le \frac{2m}{3(M+m)+2m(k+1-K_1)}.
$$
Finally, replacing in \eqref{H2} $B_k^2$ by its upper bound derived from the last display, we get
the suggested value for $h_{k+1}$.

\subsection{Extension to mixtures of strongly log-concave densities}
\label{subseq:mixt}

We describe here a simple setting in which a suitable version of the LMC algorithm
yields efficient sampling algorithm for a target function which is not log-concave.
Indeed, let us assume that
\begin{align}\label{mixture}
\pi(\btheta) = \int_{H} \pi_1(\btheta|\eeta)\,\pi_0(d\eeta),
\end{align}
where $H$ is an arbitrary measurable space, $\pi_0$ is a probability distribution
on $H$ and $\pi_1(\cdot|\cdot)$ is a Markov kernel on $\RR^p\times H$. This means
that $\pi_2(d\btheta,d\eeta) = \pi_1(\btheta|\eeta)\,\pi_0(d\eeta)d\btheta$ defines
a probability measure on $\RR^p\times H$ of which $\pi$ is the first marginal.

\begin{theorem}\label{th:three}
    Assume that $\pi_1(\btheta|\eeta) = \exp\{-f_{\eeta}(\btheta)\}$ so that
    for every $\eeta\in H$, $f_{\eeta}$ satisfies assumption \eqref{1}.
    Define the mixture LMC (MLMC) algorithm as follows: sample $\eeta\sim \pi_0$ and choose an initial value $\bvartheta_0\sim \nu_0$, then compute
    \begin{align}\label{LMC3}
    \bvartheta_{k+1}^{\rm MLMC} = \bvartheta^{\rm MLMC}_{k} - h_{k+1} \nabla
    f_{\eeta}(\bvartheta_{k}^{\rm MLMC})+ \sqrt{2h_{k+1}}\;\bxi_{k+1};
    \qquad k=0,1,2,\ldots
    \end{align}
    where $h_k$ is defined by \eqref{h2} and $\bxi_1,\ldots,\bxi_{k},\ldots$ is a sequence
    of mutually independent, and independent of $(\eeta,\bvartheta_{0})$,
    centered Gaussian vectors with covariance matrices equal to identity. It holds that,
    for every positive integer $k\ge K_1$ (see \cref{K1} for the definition of $K_1$),
    \begin{equation}
    W_2(\nu_k,\pi)\le \frac{3.5M\sqrt{p}}{m\sqrt{M+m+(\nicefrac23)m(k-K_1)}}.
    \end{equation}
\end{theorem}
This result extends the applicability of Langevin based techniques
to a wider framework than the one of strongly log-concave distributions.
The proof, postponed to \Cref{secProof}, is  a straightforward consequence of \Cref{thOneBis}.


\section{Guarantees for the inaccurate gradient version}
\label{sec:4}

In some situations, the precise evaluation of the gradient $\nabla f(\btheta)$
is computationally expensive or practically impossible, but it is possible to
obtain noisy evaluations of $\nabla f$ at any point. This is the setting considered
in the present section. More precisely, we assume that at any point $\bvartheta_{k,h}\in\RR^p$
of the LMC algorithm, we can observe the value
\begin{equation}
\bY_{k,h}  = \nabla f(\bvartheta_{k,h}) + \bzeta_{k},
\end{equation}
where $\{\bzeta_{k}:\,k=0,1,\ldots\}$ is a sequence of random (noise) vectors.
The noisy LMC (nLMC) algorithm is defined as
\begin{align}\label{9}
\bvartheta_{k+1,h} = \bvartheta_{k,h} - h \bY_{k,h}+ \sqrt{2h}\;\bxi_{k+1};\qquad k=0,1,2,\ldots
\end{align}
where $h>0$ and $\bxi_{k+1}$ are as in \eqref{2}. The noise $\{\bzeta_{k}:\,k=0,1,\ldots\}$ is assumed to satisfy the following condition.

\textbf{\textsc{Condition N:}} for some $\delta>0$ and $\sigma>0$ and for every $k\in\NN$,
\vspace{-4pt}
\begin{itemize}\itemsep=3pt
\item (bounded bias) $\bfE\big[\big\|\bfE(\bzeta_{k}|\bvartheta_{k,h})\big\|_2^2\big] \le \delta^2p$,
\item (bounded variance) $\bfE[\|\bzeta_{k}-\bfE(\bzeta_{k}|\bvartheta_{k,h})\|_2^2]\le \sigma^2 p$,
\item (independence of updates) $\bxi_{k+1}$ in \eqref{9} is independent of $(\bzeta_0,\ldots,\bzeta_{k})$.
\end{itemize}

We emphasize right away that the random vectors $\bzeta_k$ are not assumed
to be independent, as opposed to what is done in \citep{DalalyanColt}.
The next theorem extends the guarantees
of \Cref{thOne} to the inaccurate-gradient setting and to the nLMC algorithm.

\begin{theorem}\label{thTwo}
Let $\bvartheta_{K,h}$ be the $K$-th iterate of the nLMC algorithm \eqref{9} and
$\nu_K$ be its distribution. If the function $f$ satisfies condition \eqref{1}
and $h\le \nicefrac2{(m+M)}$ then
\begin{align}\label{A}
W_2(\nu_K, \pi) \le (1&-mh)^{K} W_2(\nu_0,\pi) + 1.65(M/m)(hp)^{1/2} \\
		&+ \frac{\delta\sqrt{p}}{m} + \frac{\sigma^2 (hp)^{1/2} }{1.65M + \sigma \sqrt{m}}\ .
\end{align}
\end{theorem}

To the best of our knowledge, the first result providing guarantees for sampling from a
distribution in the scenario when precise evaluations of the log-density or its gradient are
not available has been established in \citep{DalalyanColt}. Prior to that work, some asymptotic
results has been established in \citep{Alquier2016}. The closely related problem of computing an
average value with respect to a distribution, when the gradient of its log-density is known up to
an additive noise, has been studied by \cite{teh16a,Vollmer2015,Nagapetyan,chen2015convergence}.  Note that these
settings are of the same flavor as those of stochastic approximation, an active area of research
in optimization and machine learning.

As compared to the analogous result in \citep{DalalyanColt}, \Cref{thTwo} above has several
advantages. First, it extends the applicability of the result to the case of a biased noise.
In other words, it allows for $\bzeta_k$ with nonzero means. Second, it considerably relaxes
the independence assumption on the sequence $\{\bzeta_k\}$, by replacing it by the independence
of the updates. Third, and perhaps the most important advantage of \Cref{thTwo} is the improved
dependence of the upper bound on $\sigma$. Indeed, while the last term in the upper bound in \Cref{thTwo}
is $O(\sigma^2)$, when $\sigma\to 0$, the corresponding term in \citep[Th.\ 3]{DalalyanColt} is
only $O(\sigma)$.

To understand the potential scope of applicability of \Cref{thTwo}, let us consider a generic example
in which $f(\btheta)$ is the average of $n$ functions defined through independent
random variables $X_1,\ldots,X_n$:
$$
f(\btheta) = \frac1n\sum_{i=1}^n \ell(\btheta,X_i).
$$
When the gradient of $\ell(\btheta,X_i)$ with respect to parameter $\btheta$ is hard to compute,
one can replace the evaluation of $\nabla f(\bvartheta_{k,h})$ at each step $k$ by that  of $Y_k=
\nabla_{\btheta} \ell (\bvartheta_{k,h},X_{N_k})$, where $N_k$
is a random variable uniformly distributed in $\{1,\ldots, n\}$ and
independent of $\bvartheta_{k,h}$. Under suitable assumptions, this random
vector satisfies the conditions of \Cref{thTwo} with $\delta = 0$ and constant $\sigma^2$.
Therefore, if we analyze the upper bound provided by \eqref{A}, we see that the
last term, due to the subsampling, is of the same order of magnitude as the second term.
Thus, using the subsampled gradient in the LMC algorithm does not cause a significant
deterioration of the precision while reducing considerably the computational burden.

Note that \Cref{thTwo} allows to handle situations in which the approximations of the
gradient are biased. This bias is controlled by the parameter $\delta$. Such  a bias
can appear when using deterministic approximations of integrals or differentials. For
instance, in statistical models with latent variables, the gradient of the
log-likelihood has often an integral form. Such integrals can be approximated
using quadrature rules, yielding a bias term, or Monte Carlo methods, yielding a
variance term.

In the preliminary version \citep{DalalyanColt} of this work, we made a mistake by
claiming that the stochastic gradient version of the LMC, introduced in \citep{WellingT11}
and often referred to as Stochastic Gradient Langevin Dynamics (SGLD), has an error
of the same order as the non-stochastic version of it. This claim is wrong, since
when $f(\btheta) = \sum_{i=1}^n \ell (\btheta,X_i)$ with a strongly convex function
$\btheta\mapsto \ell(\btheta,x)$ and iid variables $X_1,\ldots,X_n$, we have $m$ and $M$
proportional to $n$. Therefore, choosing $Y_k = n\nabla_{\btheta} \ell (\bvartheta_{k,h},X_{N_k})$
as a noisy version of the gradient (where $N_k$ is a uniformly over $\{1,\ldots,n\}$ distributed
random variable independent of $\bvartheta_{k,h}$), we get $\delta=0$ but $\sigma^2$ proportional
to $n^2$. Therefore, the last term in \eqref{A} is of order $(nhp)^{1/2}$ and dominates the other
terms. Furthermore, replacing $Y_k$ by $Y_k = \frac{n}{s}\sum_{j=1}^s\nabla_{\btheta} \ell (\bvartheta_{k,h},X_{N^j_k})$
with iid variables $N_k^1,\ldots,N_k^s$ does not help, since then $\sigma^2$ is of order $n^2/s$ and
the last term in \eqref{A} is of order $(nhp/s)^{1/2}$, which is still  larger than the term $(M/m)(hp)^{1/2}$.
This discussion shows that \Cref{thTwo} applied to SGLD is of limited interest.
For a more in-depth analysis of the SGLD, we refer the reader to \citep{Nagapetyan, raginsky17a, Xu}.

It is also worth mentioning here that another example of approximate gradient---based on a quadratic
approximation of the log-likelihood of the generalized linear model---has been considered in
\citep[Section 5]{huggins17a}. It corresponds, in terms of condition N, to a situation in which
the variance $\sigma^2$ vanishes but the bias $\delta$ is non-zero.

An important ingredient of the proof of \Cref{thTwo} is the following simple result, which
can be useful in other contexts as well (for a proof, see \Cref{lemE} in \Cref{ssecLem} below).

\begin{lemma}\label{lemD}
Let $A$, $B$ and $C$ be given non-negative numbers such that $A\in (0,1)$. Assume that the sequence of non-negative numbers
$\{x_k\}_{k=0,1,2,\ldots}$ satisfies the recursive inequality
\begin{align}
x^2_{k+1}&\le [(1-A)x_k+C]^2+B^2
\end{align}
for every integer $k\ge 0$.  Then, for all integers $k\ge 0$,
\begin{align}
x_k&\le (1-A)^{k} x_0 + \frac{C}{A} + \frac{B^2}{C+\sqrt{A}\,B}.
\end{align}
\end{lemma}

Thanks to this lemma, the upper bound on the Wasserstein distance provided by \eqref{A} is sharper
than the one proposed in \citep{DalalyanColt}.

\section{Guarantees under additional smoothness}
\label{sec:5}

When the function $f$ has Lipschitz continuous Hessian, one can get improved rates of
convergence. This has been noted by \citep{Dalalyan14}, where the author proposed to use a modified
version of the LMC algorithm, the LMC with Ozaki discretization, in order to take
advantage of the smoothness of the Hessian. On the other hand, it has been proved
in \citep{alfonsi2014,alfonsi2015} that the boundedness of the third order derivative of $f$
(equivalent to the boundedness of the second-order derivative of the drift of the
Langevin diffusion) implies that the Wasserstein distance between the marginals
of the Langevin diffusion and its Euler discretization are of order $h\sqrt{\log(1/h)}$.
Note however, that in \citep{alfonsi2015} there is no evaluation of the impact of
the dimension on the quality of the Euler approximation. This evaluation has been done by
\cite{Durmus2} by showing that the Wasserstein error of the Euler approximation is
of order $hp$. This raises the following important question: is it possible to get
advantage of the Lipschitz continuity of the Hessian of $f$ in order to improve the
guarantees on the quality of sampling by the standard LMC algorithm. The answer of
this question is affirmative and is stated in the next theorem.

In what follows, for any matrix $\bfM$, we denote by $\|\bfM\|$ and $\|\bfM\|_F$,
respectively, the spectral norm and the Frobenius norm of $\bfM$. We write
$\bfM\preceq\bfM'$ or $\bfM'\succeq\bfM'$ to indicate that the matrix $\bfM'-\bfM$
is positive semi-definite.

\noindent\textbf{\textsc{Condition F:}} the function $f$ is twice differentiable and for some positive numbers
$m$, $M$ and $M_2$,
\vspace{-10pt}
\begin{itemize}\itemsep=3pt
\item (strong convexity) $\nabla^2 f(\btheta)\succeq m\bfI_p$, for every $\btheta\in\RR^p$,
\item (bounded second derivative) $\nabla^2 f(\btheta)\preceq M\bfI_p$, for every $\btheta\in\RR^p$,
\item (further smoothness) $\|\nabla^2 f(\btheta)-\nabla^2 f(\btheta')\|\le M_2\|\btheta-\btheta'\|_2$, for every $\btheta,\btheta'\in\RR^p$.
\end{itemize}

\begin{theorem}\label{thFour}
Let $\bvartheta_{K,h}$ be the $K$-th iterate of the nLMC algorithm \eqref{9} and
$\nu_K$ be its distribution. Assume that conditions {\bf F}  and {\bf N} are
satisfied. Then, for every  $h\le \nicefrac2{(m+M)}$, we have
\begin{align}\label{A2}
W_2(\nu_K, \pi) \le (1&-mh)^{K} W_2(\nu_0,\pi) +
		\frac{M_2 hp}{2m}+ \frac{11 M h\sqrt{Mp}}{5m}\\
		&+ \frac{\delta\sqrt{p}}{m}  + \frac{2\sigma^2 \sqrt{hp} }{M_2\sqrt{hp} + 2\sigma \sqrt{m}}.
\end{align}
\end{theorem}
In the last inequality, $11/5$ is an upper bound for $0.5+2\sqrt{2/3}\approx 2.133$.


When applying the nLMC algorithm to sample from a target density, the user may usually specify
four parameters: the step-size $h$, the number of iterations $K$, the tolerated precision $\delta$
of the deterministic approximation and the precision $\sigma$ of the stochastic approximation. An
attractive  feature of \Cref{thFour} is that the contributions of these four parameters are well
separated, especially if we upper bound the last term by $2\sigma^2/M_2$. As a consequence,
in order to have an error of order $\eps$ in Wasserstein distance, we might choose:
$\sigma$ at most of order $\sqrt{\eps}$, $\delta$ at most of order $m\eps/\sqrt{p}$,
$h$ of order $\eps/p$ and $K$ of order $(p/m\eps)\log(p/\eps)$. Akin to \Cref{thOneBis}, one can
use variable step-sizes to avoid the logarithmic factor; we leave these computations
to the reader.

Note that if we instantiate \Cref{thFour} to the case of accurate gradient evaluations, that is
when $\sigma=\delta=0$, we recover the constant step-size version of \citep[Theorem 8]{Durmus2},
with optimized constants. Indeed, for contant step-size, \citep[Theorem 8]{Durmus2} yields
\begin{align}
    W_2(\nu_K, \pi) \le \Big\{2(1-\bar m\,h)^{K} W_2(\nu_0,\pi)^2 +2ph^2 \Big(
    \frac{M^2}{\bar m} + \frac{M^4}{3m\bar m^2} + \frac{M_2^2p}{3\bar m^2}+O(h)\Big)
		\Big\}^{1/2},\label{DM2}
\end{align}
where $\bar m = \frac{mM}{m+M}\in [m/2, m)$ and the term $O(h)$ can be given explicitly.  A visual comparison of the optimal number of iterations
obtained from this bound to that obtained from
\Cref{thFour} (with $\delta=\sigma = 0$) is provided
in \Cref{fig2}.

\begin{figure}
    \centering
    \includegraphics[width = 0.6\textwidth]{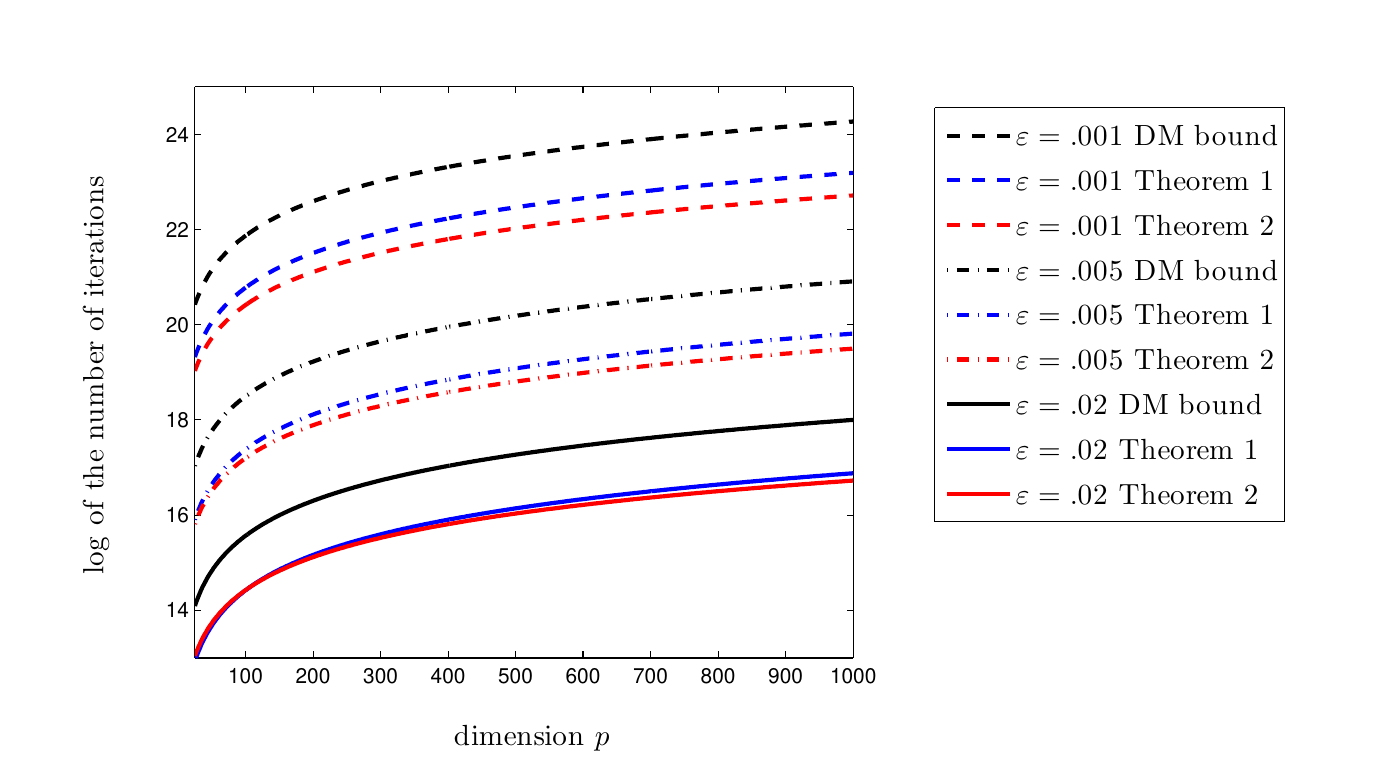}
    \caption{Plots showing the logarithm of the number of iterations as function of dimension $p$ for
    several values of $\eps$. The plotted values
    are derived from \Cref{thFour} and \eqref{DM2} (referred to as DM bound) using the data $m=10$, $M=50$, $M_2=1$,
    $W_2^2(\nu_0,\pi)=p + (p/m)$, $\delta= \sigma=0$.}
    \label{fig2}
\end{figure}

Under the assumption of Lipschitz continuity of the Hessian of $f$, one may wonder whether
second-order methods that make use of the Hessian in addition to the gradient are able
to outperform the standard LMC algorithm. The most relevant candidate algorithms for
this are the LMC with Ozaki discretization (LMCO) and a variant of it, LMCO', a slightly
modified version of an algorithm introduced in \citep{Dalalyan14}. The LMCO is a recursive
algorithm the update rule of which is defined as
follows: For every $k\ge 0$, we set $\bfH_{k} = \nabla^2 f(\bvartheta_{k,h}^{\rm LMCO})$,
which is an invertible $p\times p$ matrix since $f$ is strongly convex, and define
\begin{align}
&\bfM_k = \big(\bfI_p-e^{-h\bfH_k}\big)\bfH_k^{-1},\qquad
\bfSigma_k  = \big(\bfI_p-e^{-2h\bfH_k}\big)\bfH_k^{-1}, \\
&\bvartheta_{k+1,h}^{\rm LMCO} = \bvartheta_{k,h}^{\rm LMCO}
	-\bfM_k\nabla f\big(\bvartheta_{k,h}^{\rm LMCO}\big) +\bfSigma_k^{1/2}\bxi_{k+1},\label{update}
\end{align}
where $\{\bxi_{k}:k\in\NN\}$ is a sequence of independent random vectors distributed according
to the $\mathcal N_p(0,\bfI_p)$ distribution. The LMCO' algorithm is based on approximating the
matrix exponentials by linear functions, more precisely, for $\bfH_k' = \nabla^2 f(\bvartheta_{k,h}^{\rm LMCO'})$,
\begin{align}
\bvartheta_{k+1,h}^{\rm LMCO'} =\, &\bvartheta_{k,h}^{\rm LMCO'}-h\Big(\bfI_p-\frac12h\bfH_k'\Big)
\nabla f\big(\bvartheta_{k,h}^{\rm LMCO'}\big)\\
 &+ \sqrt{2h}\Big(\bfI_p-h\bfH_k'+\frac13h^2(\bfH_k')^2\Big)^{1/2}\bxi_{k+1}.
\label{update1}
\end{align}
Let us mention right away that the stochastic perturbation present in the last display can be
computed in practice without taking the matrix square-root. Indeed, it suffices to generate
two independent standard Gaussian vectors $\bfeta_{k+1}$ and $\bfeta_{k+1}'$; then  the random vector
\begin{align}
\big(\bfI_p-(\nicefrac12)h\bfH_k'\big)\bfeta_{k+1}+(\nicefrac{\sqrt{3}}{6})\,h \,\bfH_k'\bfeta'_{k+1}
\end{align}
has exactly the same distribution as
$\big(\bfI_p-h\bfH_k'+(\nicefrac13)h^2(\bfH_k')^2\big)^{1/2}\bxi_{k+1}$.

In the rest of this section,
we provide guarantees for methods LMCO and LMCO'. Note that we consider
only the case where the gradient and the Hessian of $f$ are computed exactly, that is without any approximation.

\begin{theorem}\label{thFive}
Let $\nu_K^{\rm LMCO}$ and $\nu_K^{\rm LMCO'}$ be, respectively, the distributions of the $K$-th iterate of the LMCO
algorithm \eqref{update} and  the  LMCO' algorithm \eqref{update1} with an initial distribution $\nu_0$.
Assume that conditions {\bf F}  and {\bf N} are  satisfied. Then, for every  $h\le \nicefrac{m}{M^2}$,
\begin{align}\label{A3'}
W_2(\nu_K^{\rm LMCO}, \pi) \le (1-0.25mh)^KW_2(\nu_{0},\pi)+ \frac{11.5 M_2h(p+1)}{m}.		
\end{align}
If, in addition, $h\le \nicefrac{3m}{4M^2}$, then
\begin{align}\label{A3''}
W_2(\nu_K^{\rm LMCO'}, \pi) &\le (1-0.25mh)^KW_2(\nu_{0},\pi)+
\frac{1.3M^2h^2\sqrt{Mp}}{m}+ \frac{7.3 M_2h(p+1)}{m}.
\end{align}
\end{theorem}

A very rough consequence of this theorem is that one has similar theoretical guarantees for
the LMCO and the LMCO' algorithms, since in most situations the middle term in the right hand side
of \eqref{A3''} is smaller than the last term. On the other hand, the per-iteration cost of the
modified algorithm LMCO' is significantly smaller than the per-iteration cost of the original LMCO.
Indeed, for the LMCO' there is no need to compute matrix exponentials neither to invert matrices,
one only needs to perform  matrix-vector multiplication for $p\times p$ matrices. Note that for many
matrices such a multiplication operation might be very cheap using the fast Fourier transform or
other similar techniques. In addition, the computational complexity of the Hessian-vector product
is provably of the same order as that of evaluating the gradient, see \citep{Griewank}. Therefore,
one iteration of the LMCO' algorithm is not more costly than one iteration of the LMC. At the same
time, the error bound \eqref{A3''} for the LMCO' is smaller than the one for the LMC provided by
\Cref{thFour}. Indeed, the term $Mh\sqrt{Mp}$ present in the bound of \Cref{thFour} is generally
of larger order than the term $(Mh)^2\sqrt{Mp}$ appearing in \eqref{A3''}.

\section{Relation with optimization}
\label{secOpt}

We have already mentioned that the LMC algorithm is very close to the gradient descent
algorithm for computing the minimum $\btheta^*$ of the function $f$.  However, when we
compare the guarantees of~\Cref{thOne} with those available for the optimization problem,
we remark the following striking difference. The approximate computation of $\btheta^*$ requires
a number of steps of the order of $\log(1/\varepsilon)$ to reach the precision $\varepsilon$,
whereas, for reaching the same precision in sampling from $\pi$, the LMC algorithm needs a
number of iterations proportional to $(p/\varepsilon^2)\log (p/\varepsilon)$.
The goal of this section is to explain that this, at first sight disappointing
behavior of the LMC algorithm is, in fact, consistent with the exponential
convergence of the gradient descent. Furthermore, the latter is obtained from the guarantees on the LMC by letting a temperature parameter go to zero.

The main ingredient for the explanation is that the function $f(\btheta)$ and the function
$f_\tau(\btheta) = f(\btheta)/\tau$ have the same point of minimum $\btheta^*$, whatever
the real number $\tau>0$. In addition, if we define the density function
$\pi_\tau(\btheta)\propto \exp\big(-f_\tau(\btheta)\big)$, then the average value
$$
\bar\btheta_\tau = \int_{\RR^p } \btheta\, \pi_\tau(\btheta)\,d\btheta
$$
tends to the minimum point $\btheta^*$ when $\tau$ goes to zero. Furthermore,
the distribution $\pi_\tau(d\btheta)$ tends to the Dirac measure at $\btheta^*$.
Clearly, $f_\tau$ satisfies \eqref{1} with the constants $m_\tau = m/\tau$
and $M_\tau = M/\tau$. Therefore, on the one hand, we can apply to $\pi_\tau$
claim (a) of \Cref{thOne}, which tells  us that if we choose $h = 1/M_\tau = \tau/M$,
then
\begin{equation}
\label{7}
W_2(\nu_K,\pi_\tau) \le \Big(1-\frac{m}{M}\Big)^K W_2(\delta_{\btheta_{0}},\pi_\tau)
+ 1.65\Big(\frac{M}{m}\Big)\Big(\frac{p\tau}{M}\Big)^{1/2}.
\end{equation}
On the other hand, the LMC algorithm with the step-size $h=\tau/M$ applied to
$f_\tau$ reads as
\begin{equation}
\label{8}
\bvartheta_{k+1,h} = \bvartheta_{k,h} - \frac1M \nabla f(\bvartheta_{k,h})+
\sqrt{\frac{2\tau}M}\;\bxi_{k+1};\qquad k=0,1,2,\ldots
\end{equation}
When the parameter $\tau$ goes to zero, the LMC sequence \eqref{8} tends
to the gradient descent sequence $\btheta_{k}$. Therefore, the limiting case
of \eqref{7} corresponding to $\tau\to 0$ writes as
\begin{equation}
\label{optimGuar}
\|\btheta^{(K)}-\btheta^*\|_2 \le \Big(1-\frac{m}{M}\Big)^K \|\btheta_{0}-\btheta^*\|_2,
\end{equation}
which is a well-known result in Optimization. This clearly shows that \Cref{thOne} is a natural
extension of the results of convergence from optimization to sampling.

Such an analogy holds true for the Newton method as well. Its counterpart in sampling is the
LMCO algorithm. Indeed, one easily checks that if $f$ is replaced by $f_\tau$ with $\tau$ going to
zero, then, for any fixed step-size $h$, the matrix $\bfSigma_k$ in \eqref{update} tends to zero.
This implies that the stochastic perturbation vanishes. On the other hand, the term
$\bfM_{k,\tau}\nabla f_\tau(\bvartheta_{k,h}^{\rm LMCO})$ tends to $\{\nabla^2 f
(\bvartheta_{k,h}^{\rm LMCO})\}^{-1}\nabla f(\bvartheta_{k,h}^{\rm LMCO})$, as $\tau\to 0$.
Thus, the updates of the Newton algorithm can be seen as the limit case, when $\tau$ goes to
zero, of the updates of the LMCO.

However, if we replace $f$ by $f_\tau$ in the upper bounds stated in \Cref{thFive} and we
let $\tau$ go to zero, we do not retrieve the well-known guarantees for the Newton method.
The main reason is that \Cref{thFive} describes the behavior of the LMCO algorithm in the
regime of small step-sizes $h$, whereas Newton's method corresponds to (a limit case of) the
LMCO with a fixed $h$. Using arguments similar to those employed in the proof of \Cref{thFive},
one can establish the following result, the proof of which is postponed to \Cref{secProof}.

\begin{proposition} \label{propB}
Let $\nu_K^{\rm LMCO}$ be the distributions of the $K$-th iterate of the LMCO
algorithm \eqref{update} with an initial distribution $\nu_0$.  Assume that condition
{\bf F} is  satisfied. Then, for every  $h>0$ and $K\in\NN$,
\begin{align}\label{A5}
W_2(\nu_K^{\rm LMCO}, \pi)	&\le \frac{2m}{M_2} \big(w_K\exp(v_Kw_K^{-2^K})\big)^{2^K}
\end{align}
with
\begin{align}
w_K &= \frac{M_2W_{2^{K+1}}(\nu_0,\pi)}{2m}+\frac12e^{-mh},\ \text{and}\
v_K = \frac{2M_2 M^{3/2}\sqrt{2p+2^{K}}}{m^3} +e^{-mh}.
\end{align}		
\end{proposition}

If we replace in the right hand side of \eqref{A5} the quantities $m$, $M$ and $M_2$, respectively,
by $m_\tau = m/\tau$, $M_\tau = M/\tau$ and $M_{2,\tau} = M_2/\tau$, and we let $\tau$ go to zero,
then it is clear that the term $v_K$ vanishes. On the other hand, if $\nu_0$ is the Dirac mass at some
point $\btheta_0$, then $w_K$ converges to $M_2\|\btheta_0-\btheta^*\|_2/(2m)$. Therefore, for Newton's
algorithm as a limiting case of \eqref{A5} we get
\begin{align}
\|\btheta_K^{\rm Newton}-\btheta^*\|_2 \le \frac{2m}{M_2} \bigg(\frac{M_2\|\btheta_0-\btheta^*\|_2}{2m}\bigg)^{2^K}.
\end{align}
The latter provides the so called quadratic rate of convergence, which is a well-known result that can be found in
many textbooks; see, for instance, \cite[Theorem 9.1]{ChongZak}.

A particularly promising remark made in \Cref{subseq:mixt} is that all the results
established for the problem of approximate sampling from a log-concave distribution
can be carried over the distributions that can be written as a mixture of (strongly)
log-concave distributions. The only required condition is to be able to sample from
the mixing distribution. This provides a well identified class of (posterior)
distributions for which the problem of finding the mode is difficult (because of
nonconvexity) whereas the sampling problem can be solved efficiently.

There are certainly other interesting connections to uncover between sampling and
optimization. In particular,  in \cite{ma2018sampling}, it was shown that in the case of mixture distributions, sampling algorithms scale linearly with the model dimension, as opposed to those of optimization, which have exponential scaling. One can think of lower bounds for sampling or finding a sampling counterpart
of Nesterov acceleration. Some recent advances on the gradient flow \Citep{Wibisono}
might be useful for achieving these goals.

\section{Conclusion}

We have presented easy-to-use finite-sample guarantees for sampling from a strongly log-concave
density using the Langevin Monte-Carlo algorithm with a fixed step-size and extended it to the
case where the gradient of the log-density can be evaluated up to some error term. Our results cover
both deterministic and random error terms. We have also demonstrated that if the log-density $f$ has
a Lipschitz continuous second-order derivative, then one can choose a larger step-size and obtain
improved convergence rate.

We have also uncovered some analogies between sampling and optimization. The underlying principle
is that an optimization algorithm may be seen as a limit case of a sampling algorithm. Therefore,
the results characterizing the convergence of the optimization schemes should have their counterparts
for sampling strategies. We have described these analogues for the steepest gradient descent and for
the Newton algorithm. However, while in the optimization the relevant characteristics of the problem
are the dimension $p$, the desired accuracy $\eps$ and the condition number $M/m$, the problem sampling
involves an additional characteristic which is the scale given by the strong-convexity constant $m$. Indeed,
if we increase $m$ by keeping the condition number $M/m$ constant, the number of iterations for the LMC to
reach the precision $\epsilon$ will decrease. In this respect, we have shown that the LMC with
Ozaki discretization, termed LMCO, has a better dependence on the overall scale of $f$ than the original
LMC algorithm. However, the weakness of the LMCO is the high computational cost of each iteration.
Therefore, we have proposed a new algorithm, LMCO', that improves the LMC in terms of its dependence on the scale
and each iteration of LMCO' is computationally much cheaper than each iteration of the LMCO.

Another interesting finding is that, in the case of accurate gradient evaluations (\textit{i.e.}, when
there is no error in the gradient computation), a suitably chosen variable step-size leads to logarithmic
improvement in the convergence rate of the LMC algorithm.

Interesting directions for future research are establishing lower bounds in the spirit of those existing
in optimization, obtaining user-friendly guarantees for computing the posterior mean or for sampling from
a non-smooth density. Some of these problems have already been tackled in several papers mentioned in
previous sections, but we believe that the techniques developed in the present work might be helpful for
revisiting and deepening the existing results.

\section{Proofs}
\label{secProof}

The basis of the proofs of all the theorems stated
in previous sections is a recursive inequality that upper bounds the error at the step $k+1$,
$W_2(\nu_{k+1},\pi)$, by an expression involving the error of the previous
step, $W_2(\nu_k,\pi)$.  To this end, we use the fact that for a suitably chosen Langevin diffusion, $\bL$, in stationary regime, we have
$W_2(\nu_k,\pi)^2 = \bfE[\|\vartheta_{k} -\bL_{kh}\|_2^2]$ and $W_2(\nu_{k+1},\pi)^2 \le \bfE[\|\vartheta_{k+1} -\bL_{(k+1)h}\|_2^2]$. The goal is then to upper bound the latter by an expression that involves the former and some
suitably controlled remainder terms. This leads to a recursive inequality and the last step of the proof is to unfold the recursion. Since different chains $\vartheta_{k,h}$ are considered in this paper, we get different recursive inequalities. \Cref{lemE} and \Cref{lemH} are the new technical tools that are used for solving the encountered recursive inequalities.  The remainder terms appearing
in the recursive inequalities are evaluated by using stochastic calculus and the smoothness
properties of $f$. The main building blocks
for these evaluations are \Cref{lemB}, \Cref{lemC} and \Cref{lemF}, the latter being
used only in the results assuming the Hessian-Lipschitz condition.

We will also make repeated use of the Minkowski inequality
and its integral version
\begin{align}\label{eqA}
\bigg\{\bfE\bigg[\bigg(\int_a^b X_t\,dt\bigg)^p\bigg]\bigg\}^{1/p} \le
\int_a^b\big\{\bfE\big[|X_t|^p\big]\big\}^{1/p} \,dt, \qquad \forall p\in \NN^*,
\end{align}
where $X$ is a random process almost all paths of which are integrable over
the interval $[a,b]$. Furthermore, for any random vector $\bX$, we define the norm
$\|\bX\|_{L_2} = (\bfE[\|\bX\|_2^2])^{1/2}$.

The next result is the central ingredient of the proofs of \Cref{thOne,thOneBis,thTwo}. Readers
interested only in the proof of \Cref{thOne,thOneBis}, are invited---in the next proof---to consider
the random vectors $\bzeta_k$ as equal to $\mathbf 0$ and $\bY_{k,\bh}$ as equal to
$\nabla f(\bvartheta_{k,\bh})$. This implies, in particular, that $\sigma = \delta =0$.

\begin{proposition}\label{propA}
Let us introduce $\varrho_{k+1}= \max(1-mh_{k+1}, Mh_{k+1}-1)$ (since $h\in(0,\nicefrac2M)$,
this value $\varrho$ satisfies $0< \varrho <1$). If $f$ satisfies \eqref{1} and $h_{k+1}\le 2/M$, then
\begin{align}
W_2(\nu_{k+1},\pi)^2
		&\le \big\{\varrho_{k+1}W_2(\nu_{k},\pi) +\alpha M(h_{k+1}^3p)^{1/2} + h_{k+1}\delta\sqrt{p}\big\}^2
		+ \sigma^2 h_{k+1}^2 p,\label{DD}
\end{align}
with $\alpha = 7\sqrt{2}/6\le 1.65$.
\end{proposition}

\begin{proof}
To simplify notation, and since there is no risk of confusion,
we will write $h$ instead of $h_{k+1}$.
The main steps of the
proof are the following. We use a synchronous coupling for approximating the distribution of the LMC sequence by that of a continuous-time Langevin diffusion. We then take advantage
of the strong convexity of $f$ for showing that, for $h$ small
enough, the error at round $k+1$ is upper bounded, up to a
additive remainder term, by the error at round k multiplied by
a factor strictly smaller than one, see \Cref{lemA}. The
smoothness of the gradient of $f$ ensures that the
aforementioned remainder term is small, see \Cref{lemB}
and \Cref{lemC} below.

Let $\bL_0$ be a random
vector drawn from $\pi$ such that
$W_2(\nu_k,\pi) = \|\bL_0-\bvartheta_{k,\bh}\|_{L_2}$ and $\bfE[\bzeta_{k}|\bvartheta_{k,\bh},\bL_0] =
\bfE[\bzeta_{k}|\bvartheta_{k,\bh}]$. Let $\bW$ be a
$p$-dimensional Brownian Motion independent of $(\bvartheta_{k,\bh}, \bL_0,\bzeta_{k})$,
such that $\bW_{h} = \sqrt{h}\,\bxi_{k+1}$. We define the stochastic process $\bL$
so that
\begin{align}\label{B}
\bL_t &= \bL_0 - \int_0^t  \nabla f(\bL_s)\,ds + \sqrt{2}\,\bW_t,\qquad\forall\, t>0.
\end{align}
It is clear that this equation implies that
\begin{align}
\ds\bL_{h}
	&= \bL_{0} - \int_{0}^{h} \nabla f(\bL_s)\,ds + \sqrt{2}\,\bW_{h}\\
	&= \bL_{0} - \int_{0}^{h} \nabla f(\bL_s)\,ds + \sqrt{2h}\,\bxi_{k+1}.
\end{align}
Furthermore, $\{\bL_t:t\ge 0\}$ is a diffusion process having $\pi$ as the stationary
distribution. Since the initial value $\bL_0$ is drawn from $\pi$, we have $\bL_t\sim \pi$
for every $t\ge 0$.

Let us denote $\bDelta_k = \bL_{0}-\bvartheta_{k,\bh}$ and
$\bDelta_{k+1} = \bL_{h}-\bvartheta_{k+1,\bh}$. We have
\begin{align}
\bDelta_{k+1}
	& = \bDelta_k  + h \bY_{k,\bh} - \int_0^h\nabla f(\bL_t)\,dt \\
	& = \bDelta_k  - h\big(\underbrace{\nabla f(\bvartheta_{k,\bh}+\bDelta_k)
			-\nabla f(\bvartheta_{k,\bh})}_{:=\bU}\big)+ h\bzeta_{k}\\
	&\qquad		-\underbrace{\int_0^h\big(\nabla f(\bL_t) - \nabla f(\bL_{0})\big)\,dt}_{:=\bV}.\label{D_}
\end{align}
Using the equalities
$\bfE[\bzeta_{k}|\bDelta_k,\bU,\bV] = \bfE[\bzeta_{k}|\bvartheta_{k,\bh},\bL_0,\bW] =
\bfE[\bzeta_{k}|\bvartheta_{k,\bh},\bL_0] = \bfE[\bzeta_{k}|\bvartheta_{k,\bh}]$, we get
\begin{align}
\|\bDelta_{k+1} \|_{L_2}^2
		&= \big\|\bDelta_k  -h \bU -\bV+h\bfE[\bzeta_{k}|\bvartheta_{k,\bh}]\big\|_{L_2}^2 +
		h^2\big\|\bzeta_{k}-\bfE[\bzeta_{k}|\bvartheta_{k,\bh}]\big\|_{L_2}^2\\
		&\le \big\|\bDelta_k  -h \bU -\bV+h\bfE[\bzeta_{k}|\bvartheta_{k,\bh}]\big\|_{L_2}^2 + \sigma^2 h^2 p\\
		&\le \big\{\|\bDelta_k  -h \bU\|_{L_2} +h\delta\sqrt{p} +\|\bV\|_{L_2}\big\}^2
			+ \sigma^2 h^2 p.\label{D}
\end{align}

We need now three technical lemmas.
\Cref{lemA} and \Cref{lemB} are borrowed from \citep{DalalyanColt}, whereas
\Cref{lemC} is an improved version of \citep[Lemma 3]{DalalyanColt}.  For the sake of self-containedness, we provide  proofs of these lemmas in \Cref{ssecLem}.

\begin{lemma}\label{lemA}
    Let $f$ be $m$-strongly convex and the gradient of $f$ be
    Lipschitz with constant $M$. If $h<2/M$, then the mapping
    $(\bfI_p-h\nabla f)$ is a contraction in the sense that
    \begin{equation}\label{ineq_lemma2}
        \big\|\bx-\by-h\big(\nabla f(\bx)-\nabla f(\by)\big)        \big\|_2 \le \big\{(1-mh)\vee (Mh-1)\big\}\|\bx-\by\|_2,
    \end{equation}
    for all $\bx,\by\in\RR^p$. In particular, using notations in \eqref{D_}, it holds that
    $\|\bDelta_k  -h \bU\|_2 \le  \varrho\|\bDelta_k\|_2$.
\end{lemma}

\begin{lemma}\label{lemB}
    If the function $f$ is continuously differentiable and
    the gradient of $f$ is Lipschitz with constant $M$, then
    $\int_{\RR^p} \|\nabla f(\bx)\|_2^2\,\pi(\bx)\,d\bx \le Mp$.
\end{lemma}

\begin{lemma}\label{lemC}
If the function $f$ and its gradient is Lipschitz with constant $M$,
$\bL$ is the Langevin diffusion \eqref{B} and $\bV(a) =
\int_a^{a+h}\big(\nabla f(\bL_t)-\nabla f(\bL_a)\big)\,dt$ for some $a\ge 0$, then
\begin{align}		
\|\bV(a)\|_{L_2}	&\le \frac12\big(h^4 M^{3}p\big)^{1/2} + \frac23(2h^3p)^{1/2}M .
\end{align}
\end{lemma}

Using \Cref{lemA} and \Cref{lemC} above, as well as the inequality $W_2(\nu_{k+1},\pi)^2\le \bfE[\|\bDelta_{k+1}
\|_2^2]$, we get the recursion
\begin{align}
W_2(\nu_{k+1},\pi)^2
		&\le \big\{\varrho W_2(\nu_{k},\pi) +
		(\nicefrac12)\big(h^4 M^{3}p\big)^{1/2} + (\nicefrac23)(2h^3p)^{1/2}M + h\delta\sqrt{p}\big\}^2
		+ \sigma^2 h^2 p\\
		&\stackrel{(a)}{\le}
		\big\{\varrho W_2(\nu_{k},\pi) +
		(\nicefrac12)\big(2h^3 M^{2}p\big)^{1/2} + (\nicefrac23)(2h^3p)^{1/2}M + h\delta\sqrt{p}\big\}^2
		+ \sigma^2 h^2 p\\
		&\stackrel{(b)}{\le}
		\big\{\varrho W_2(\nu_{k},\pi) + \alpha M\big(h^3 p\big)^{1/2} + h\delta\sqrt{p}\big\}^2
		+ \sigma^2 h^2 p,\label{DD}
\end{align}
where in $(a)$ we have used the condition $h\le 2/M$ whereas in $(b)$
we have put $\alpha = 7\sqrt{2}/6\le 1.65$.
\end{proof}

\subsection{Proof of \Cref{thOne}}

Using \Cref{propA} with $\sigma=\delta =0$, we get
$W_2(\nu_{k+1},\pi)\le \varrho\, W_2(\nu_{k},\pi) + \|\bV\|_{L_2}$ for all $k\in\NN$.
In view of \Cref{lemC}, this yields
\begin{align}
W_2(\nu_{k+1},\pi)
		\le \varrho\, W_2(\nu_{k},\pi) + \alpha M (h^3 p)^{1/2}.
\end{align}
Using this inequality repeatedly for $k+1,k,k-1,\ldots,1$, we get
\begin{align}
W_2(\nu_{k+1},\pi)
		&\le \varrho^{k+1}\, W_2(\nu_{0},\pi) + \alpha M (h^3 p)^{1/2}(1+\varrho+\ldots+\varrho^k)\\
		&\le \varrho^{k+1}\, W_2(\nu_{0},\pi) + \alpha M (h^3 p)^{1/2}(1-\varrho)^{-1}.
\end{align}
This completes the proof.

\subsection{Proof of \Cref{thOneBis}}
Recall that $\alpha = 7\sqrt{2}/6\le 1.65$.
\Cref{thOne} implies that using the step-size $h_k =2/(M+m)$ for $k=1,\ldots,K_1$, we get
\begin{align}
W_2(\nu_{K_1},\pi)
		&\le \Big(1+\frac{2m}{M-m}\Big)^{-K_1} W_2(\nu_0,\pi) +
			\frac{\alpha M}{m}\Big(\frac{2p}{m+M}\Big)^{1/2}\\
		&\le \frac{3.5M}{m}\Big(\frac{p}{M+m}\Big)^{1/2}.\label{F1}
\end{align}
Starting from this iteration $K_1$, we use a decreasing step-size
\begin{align}\label{H1}
h_{k+1}  = \frac{2}{M+m+(\nicefrac23)m(k-K_1)}.
\end{align}
Let us show by induction over $k$ that
\begin{align}
W_2(\nu_k,\pi)
		&\le \frac{3.5M}{m}\bigg(\frac{p}{M+m+(\nicefrac23)m(k-K_1)}\bigg)^{1/2},\qquad \forall\,k\ge K_1.\label{D5}
\end{align}
For $k=K_1$, this inequality is true in view of \eqref{F1}. Assume now that \eqref{D5}  is true for some $k$.
For $k+1$, we have
\begin{align}
W_2(\nu_{k+1},\pi)
	& \le (1-mh_{k+1}) W_2(\nu_{k},\pi) + \alpha M\sqrt{p}\; h_{k+1}^{3/2}\\
	& \le (1-mh_{k+1}) \frac{3.5M\sqrt{p}\, (h_{k+1}/2)^{1/2}}{m} + \alpha M\sqrt{p}\; h_{k+1}^{3/2}\\
	& \le (1-\frac13mh_{k+1}) \frac{3.5M\sqrt{p}\, (h_{k+1}/2)^{1/2}}{m}.
\end{align}
One can check that
\begin{align}
(1-\frac13mh_{k+1})(h_{k+1}/2)^{1/2}
		& = \frac{\sqrt{3}\,[m+3M+2m(k-K_1)]}{[3m+3M+2m(k-K_1)]^{3/2}} \\
		& \le \frac{\sqrt{3}\,[m+3M+2m(k-K_1)]^{1/2}}{3m+3M+2m(k-K_1)} \\
		& \le \frac{\sqrt{3}}{[3m+3M+2m(k+1-K_1)]^{1/2}}.
\end{align}
This completes the proof of the theorem.

\subsection{Proof of \Cref{th:three}}
    Let us denote by $\nu_k(\cdot|\bx)$ the conditional
    distribution of $\vartheta^{\rm MLMC}_k$ given
    $\eeta=\bx$. In view of \Cref{thOneBis}, we have
    \begin{equation}
        W_2\big(\nu_k(\cdot|\bx),\pi_1(\cdot|\bx)\big)
        \le \frac{3.5M\sqrt{p}}{m\sqrt{M+m+(\nicefrac23)m(k-K_1)}},\qquad \forall\bx\in H.
    \end{equation}
    This readily yields
    \begin{equation}
        \int_H W_2\big(\nu_k(\cdot|\bx),\pi_1(\cdot|\bx)\big)\,\pi_0(d\bx)
        \le \frac{3.5M\sqrt{p}}{m\sqrt{M+m+(\nicefrac23)m(k-K_1)}}.
    \end{equation}
    The last step is to apply the convexity of the Wasserstein distance,
    which means that for any probability measure $\pi_0$, we have
    $$
    \int_H W_2\big(\nu_k(\cdot|\bx),\pi_1(\cdot|\bx)\big)\,\pi_0(d\bx)
    \ge
    W_2\bigg(\int_H\nu_k(\cdot|\bx)\,\pi_0(d\bx),
    \int_H\pi_1(\cdot|\bx)\,\pi_0(d\bx)\bigg) =
    W_2(\nu_k,\pi).
    $$

\subsection{Proof of \Cref{thTwo}}

As explained in \Cref{sec:4}, the main new ingredient of the proof is
\Cref{lemD}, that has to be combined with \Cref{propA}. We postpone the proof
of \Cref{lemD} to \Cref{ssecLem} and do it in a more general form (see \Cref{lemE}).

In view of \Cref{propA}, we have
\begin{align}
W_2(\nu_{k+1},\pi)^2
		&\le \big\{(1-mh)W_2(\nu_{k},\pi) +\alpha M(h^3p)^{1/2} + h\delta\sqrt{p}\big\}^2
		+ \sigma^2 h^2 p.
\end{align}
We apply now \Cref{lemD} with $A=mh$, $B=\sigma h\sqrt{p}$ and $C = \alpha M(h^3p)^{1/2}+
h\delta\sqrt{p}$, which implies that $W_2(\nu_{k},\pi) $ is less than or equal to
\begin{align}
(1-mh)^kW_2(\nu_0,\pi) + \frac{\alpha M(hp)^{1/2}+
\delta\sqrt{p}}{m} + \frac{\sigma^2 h \sqrt{p}}{\alpha Mh^{1/2}+
\delta+(mh)^{1/2}\, \sigma }.
\end{align}
This completes the proof of the theorem.

\subsection{Proof of \Cref{thFour}}

Using the same construction and the same definitions as in the proof of \Cref{propA}, for
$\bDelta_k = \bL_{0}-\bvartheta_{k,\bh}$, we have
\begin{align}
\bDelta_{k+1} - \bDelta_k
	& =  h \bY_{k,\bh} - \int_{I_k}\nabla f(\bL_t)\,dt \\
	& = - h\big(\underbrace{\nabla f(\bvartheta_{k,\bh}+\bDelta_k)-\nabla f(\bvartheta_{k,\bh})
			}_{:=\bU}\big) \\
	&\qquad	- \sqrt2\,\underbrace{\int_0^h \int_0^t \nabla^2 f(\bL_s)d\bW_s\,dt}_{:=\bS} +
		h\bzeta_{k}\\
	&\qquad
			-\underbrace{\int_0^h\big(\nabla f(\bL_t) - \nabla f(\bL_{0})-\sqrt2\,\int_0^t\nabla^2
			f(\bL_s)d\bW_s\big)\,dt}_{:=\bar\bV}.
\end{align}
Using the following equalities of conditional expectations $\bfE[\bzeta_{k}|\bDelta_k,\bU,\bar\bV] =
\bfE[\bzeta_{k}|\bvartheta_{k,\bh},\bL_0,\bW] = \bfE[\bzeta_{k}|\bvartheta_{k,\bh},\bL_0] =
\bfE[\bzeta_{k}|\bvartheta_{k,\bh}]$ and $\bfE[\bS_h|\bvartheta_{k,\bh},\bL_0]=0$, we get
\begin{align}
\|\bDelta_{k+1} \|_{L_2}^2
		&\leq \big\|\bDelta_k  -h \bU -\bar\bV-\sqrt{2}\bS_h+
				h\bfE[\bzeta_{k}|\bvartheta_{k,\bh}]\big\|_{L_2}^2 + \sigma^2 h^2 p\\
		&\le \big\{\big(\|\bDelta_k  -h \bU\|_{L_2}^2 + 2 \|\bS_h\|_{L_2}^2\big)^{1/2} +
			h\delta\sqrt{p} +\|\bar\bV\|_{L_2}\big\}^2 + \sigma^2 h^2 p.\label{D2}
\end{align}
In addition, we have
\begin{align}
\|\bS_h\|_{L_2}^2
		& = \Big\|\int_0^h (h-s)\nabla^2 f(\bL_s)\,d\bW_s\Big\|_{L_2}^2\\
		& = \int_0^h (h-s)^2\bfE[\|\nabla^2 f(\bL_s)\|_F^2]\,ds
		\le (\nicefrac13)\, M^2 h^3 p.
\end{align}
Setting $x_k = \|\bDelta_{k} \|_{L_2} = W_2(\nu_{k},\pi)$ and using \Cref{lemA}, this yields
\begin{align}
x_{k+1}^2
		&\le \big\{\big((1-mh)^2x_k^2 + (\nicefrac23)\, M^2 h^3 p\big)^{1/2} +
			h\delta\sqrt{p} + \|\bar\bV\|_{L_2}\big\}^2 + \sigma^2 h^2 p.\label{D3}
\end{align}
Let us define $A = mh$,
$F =  (\nicefrac23)\, M^2 h^3 p$, $G = \sigma^2 h^2 p$ and\footnote{In view of \Cref{lemF} in
\Cref{ssecLem}, we have $h\delta\sqrt{p}+\|\bar\bV\|_{L_2}\le C$.}
$$
C = h\delta\sqrt{p} + 0.5 M_2h^2 p + 0.5 M^{3/2} h^2\sqrt{p}.
$$
Then
$$
x_{k+1}^2\le \big\{\big((1-A)^2x_k^2 + F\big)^{1/2} + C\big\}^2 + G.
$$			
One can deduce from this inequality that
$x_{k+1}^2\le \big((1-A)x_k + C\big)^{2} + F + G + 2 C\sqrt{F}$. Therefore,
using \eqref{inq3} of \Cref{lemE} below, we get
\begin{align}
x_{k}
		&\le (1-A)^k x_0 + \frac{C}{A} + \frac{F+G+2C\sqrt{F}}{C+\big(A(F+G+2C\sqrt{F})\big)^{1/2}}\\
		&\le (1-A)^k x_0 + (C/A) + 2(F/A)^{1/2} + \frac{G}{C+\sqrt{AG}}.
\end{align}
Replacing $A,C,F$ and $G$ by their respective expressions, we get the claim of the theorem.

\subsection{Proof of \Cref{thFive}}

To ease notation, throughout this proof, we will write $\nu_k$ and $\nu_k'$ instead of
$\nu_k^{\rm LMCO}$ and $\nu_k^{\rm LMCO'}$, respectively.

Let $\bD_0\sim \nu_k$ and $\bL_0\sim \pi$
be two random  variables such that $\|\bD_0-\bL_0\|_{L_2}^2= W_2(\nu_k,\pi)$. Let $\bW$ be a $p$-dimensional
Brownian motion independent of $(\bD_0,\bL_0)$. We define $\bL$ to be the Langevin
diffusion process \eqref{B} driven by $\bW$ and starting at $\bL_0$, whereas $\bD$ is the
process starting at $\bD_0$ and satisfying the stochastic differential equation
\begin{equation}\label{D-ornstein}
d\bD_t = -[\nabla f(\bD_0) + \nabla^2 f(\bD_0)(\bD_t-\bD_0)]\,dt+\sqrt{2}\,d\bW_t,\quad t\ge 0.
\end{equation}
This is an Ornstein-Uhlenbeck process. It can be expressed explicitly as a function of
$\bD_0$ and $\bW$. The corresponding expression implies that $\bD_h\sim \nu_{k+1}$ and, hence,
$W_2(\nu_{k+1},\pi)\le \|\bD_h-\bL_h\|_{L_2}^2$.

An important ingredient of our proof is the following version of the Gronwall lemma,
the proof of which is postponed to \Cref{ssecLem}.

\begin{lemma}\label{lemG}
Let $\balpha:[0,T]\times\Omega\to \RR^p$ be a continuous semi-martingale and
$\bfH:[0,T]\times\Omega\to\RR^{p\times p}$ be a random process with continuous paths
in the space of all symmetric $p\times p$ matrices such that $\bfH_s\bfH_t = \bfH_t\bfH_s$
for every $s,t\in[0,T]$. If $\bx:[0,T]\times\Omega\to \RR^p$ is a semi-martingale satisfying
the identity
\begin{align}\label{recx}
\bx_t = \balpha_t -\int_0^t \bfH_s\bx_s\,ds,\qquad \forall t\in[0,T],
\end{align}
then, for every $t\in[0,T]$,
\begin{align}\label{finalx}
\bx_t = \exp\Big\{-\int_0^t\bfH_s\,ds\Big\} \balpha_0 + \int_0^t \exp\Big\{-\int_s^t\bfH_u\,du\Big\}d\balpha_s.
\end{align}
\end{lemma}

We denote $\bX_t = \bL_t -\bL_0- (\bD_t-\bD_0)$, where $\bD_t$ is the random process  defined in \eqref{D-ornstein}
and $\bL_t$ is the Langevin diffusion driven by the same Wiener process $\bW$ and with initial
condition $\bL_0\sim\pi$. It is clear that
\begin{align}
\bX_t
		&= -\int_0^t \nabla f(\bL_s)\,ds+\int_0^t[\nabla f(\bD_0) + \nabla^2 f(\bD_0)(\bD_s-\bD_0)]\,ds\\
		&= -\int_0^t \big\{\nabla f(\bL_s)-\nabla f(\bD_0)-\nabla^2 f(\bD_0)(\bL_s-\bL_0)\big\}\,ds
		-\int_0^t \nabla^2 f(\bD_0)\bX_s\,ds.
\end{align}
Using \Cref{lemG}, we get
\begin{align}
\bX_t
		&= -\int_0^t e^{-s\nabla^2 f(\bD_0)} \big\{\nabla f(\bL_s)-\nabla f(\bD_0)-\nabla^2 f(\bD_0)(\bL_s-\bL_0)\big\}\,ds\\
		&=  \int_0^t e^{-s\nabla^2 f(\bD_0)}\,ds [\nabla f(\bD_0)-\nabla f(\bL_0)]\\
		&\qquad -\int_0^t e^{-s\nabla^2 f(\bD_0)}
		\big\{\nabla f(\bL_s)-\nabla f(\bL_0)-\nabla^2 f(\bL_0)(\bL_s-\bL_0)\big\}\,ds\\
		&\qquad -\int_0^t e^{-s\nabla^2 f(\bD_0)}
		[\nabla^2 f(\bD_0)-\nabla^2 f(\bL_0)]\int_0^s \nabla f(\bL_u)\,du\,ds\\
		&\qquad +\sqrt{2}\int_0^t e^{-s\nabla^2 f(\bD_0)}
		[\nabla^2 f(\bD_0)-\nabla^2 f(\bL_0)]\bW_s\,ds.		\label{fourint}
\end{align}
Let us set $\bDelta_t = \bL_t-\bD_t$. We have $\bX_t = \bDelta_t-\bDelta_0 = A_t-B_t-C_t+S_t$,
where $A_t$, $B_t$, $C_t$ and $S_t$ stand for the four integrals in \eqref{fourint}. We now evaluate
these terms separately. For the first one, using the notation $\bfH_0 = \nabla^2 f(\bD_0)$ and
the identity $\nabla f(\bL_0)-\nabla f(\bD_0) = \int_0^1 \nabla^2 f(\bD_0+x\bDelta_0)\,dx \bDelta_0$,
we get
\begin{align}
\|\bDelta_0+A_t\|_2
		&\le \|\bDelta_0-t\big(\nabla f(\bL_0)-\nabla f(\bD_0)\big)\|_2 \\
		&\qquad+ \int_0^t\|\bfI-e^{-s\bfH_0}\|\,ds
		\big\|\nabla f(\bL_0)-\nabla f(\bD_0)\big\|_2\\
		&\le (1-mt + 0.5M^2t^2)\|\bDelta_0\|_2.\label{At}
\end{align}
For the term $B_t$ with $t\le h\le m/M^2\le 1/M$, we can apply \eqref{eqB} to infer that
\begin{align}
\|B_t\|_{L_2}^2
		&\le 0.88 M_2t^2(p^2+2p)^{1/2}.\label{Bt}
\end{align}
As for $C_t$, in view of the inequality $\|\nabla^2 f(\bL_0)-\nabla^2 f(\bD_0)\|\le
M_2\|\bDelta_0\|_2\wedge M\le \sqrt{MM_2\|\bDelta_0\|_2}$, we have
\begin{align}
\|C_t\|_2
		&\le \sqrt{MM_2\|\bDelta_0\|_2}\int_0^t \int_0^s\|\nabla f(\bL_u)\|_2\,du\,ds\\
		&\le \mu\|\bDelta_0\|_2 + (4\mu)^{-1}MM_2\bigg(\int_0^t (t-u)\|\nabla f(\bL_u)\|_2\,du\bigg)^2.
\end{align}
On the other hand, the fact that $\bfE[\|\nabla f(\bL_u)\|_2^4]\le M^2(p^2+2p)$ yields
\begin{align}\label{eq15}
\bigg(\int_0^t (t-u)(\bfE[\|\nabla f(\bL_u)\|_2^4])^{1/4}\,du\bigg)^2 \le \frac{Mt^4(p^2+2p)^{1/2}}{4}.
\end{align}
This implies the inequality
\begin{align}
\|C_t\|_{L_2}
		&\le \mu W_2(\nu_k,\pi) + (16\mu)^{-1}M^2M_2t^4(p+1). \label{Ct}
\end{align}
Finally, using the integration by parts formula for semi-martingales, one can easily write $S_t$
as a stochastic integral with respect to $\bW$ and derive from that representation the inequality
\begin{align}
\|S_t\|_{L_2}^2
		&\le 2\bfE\bigg[\int_0^t \bigg\|\int_u^te^{-s\bfH_0}\,ds\big(\nabla^2 f(\bL_0)-\nabla^2 f(\bD_0)\big)\bigg\|_F^2\,du\bigg]\\
		&\le 2p\bfE[(M_2\|\bDelta_0\|_2\wedge M)^2] \int_0^t (t-u)^2\,du \le (\nicefrac23)M_2Mpt^3\|\bDelta_0\|_{L_2}^2.
		\label{St}
\end{align}
Putting all these pieces together, taking the expectation, using the Minkowski inequality,
the equality $\bfE[(\bDelta_0+A_h)^\top S_h]=0$ and the inequality $\sqrt{a^2+b}\le a+b/(2a)$,  we get
\begin{align}
\|\bDelta_h\|_{L_2}^2
		&= \|\bDelta_0 +A_h-B_h-C_h+S_h\|_{L_2}^2\\
		&\le \big(\|\bDelta_0+A_h\|_{L_2}^2+\|S_h\|_{L_2}^2\big)^{1/2}+\|B_h\|_{L_2}^2+\|C_h\|_{L_2}^2\\
		&\le \big(1-mh + 0.5M^2h^2+\mu\big)\|\bDelta_0\|_{L_2}^2+ \frac{M_2Mph^3}{3(1-mh + 0.5M^2h^2)}\\
		&\qquad + 0.88 M_2h^2(p^2+2p)^{1/2} + \frac{M^2M_2h^4}{16\mu}(p+1).\label{eq12'}
\end{align}
Let $\mu$ be any real number smaller than $0.5h(m - 0.5M^2h)$;
Eq.\ \eqref{eq12'} and the inequality $p^2+2p\le (p+1)^2$ yield
\begin{align}
W_2(\nu_{k+1},\pi)
		&\le (1 - \mu)W_2(\nu_{k},\pi)+ \frac{M_2Mph^3}{3(1-2\mu)}+ 0.88 M_2h^2(p+1) \\
		&\qquad + \frac{M^2M_2h^4}{16\mu}(p+1).\label{eq13'}
\end{align}
Since $h\le m/M^2$, we can choose $\mu = 0.25mh$ so that $1-2\mu = 1-0.5 mh \ge 0.5$ and
\begin{align}
W_2(\nu_{k+1},\pi)
		&\le (1 - 0.25mh)W_2(\nu_{k},\pi)+ \frac{2M_2Mph^3}{3}+ 0.88 M_2h^2(p+1)\\
		&\qquad +
		\frac{M^2M_2h^3}{4m}(p+1)\\
		&\le (1 - 0.25mh)W_2(\nu_{k},\pi)+ 1.8 M_2h^2(p+1).\label{eq13}
\end{align}
This recursion implies the inequality
\begin{align}
W_2(\nu_{k},\pi)
		&\le (1-0.25mh)^kW_2(\nu_{0},\pi)+ \frac{1.8M_2h(p+1)}{0.25m}\\
		&= (1-0.25mh)^kW_2(\nu_{0},\pi)+ \frac{7.2 M_2h(p+1)}{m}.
\end{align}
This completes the proof of claim \eqref{A3'} of the theorem.

To establish inequality \eqref{A3''}, we follow the same steps as in the proof of \eqref{A3'}, with a slightly different
choice of the process $\bD$. More precisely, we define $\bD$ by
\begin{align}
\bD_t -\bD_0 &= -(t\bfI_p-0.5t^2\nabla^2 f(\bD_0))\nabla f(\bD_0) + \sqrt{2}\int_0^t(\bfI-(t-u)\nabla^2 f(\bD_0))\,d\bW_u.
\end{align}
One can check that the conditional distribution of $\bD_h$ given $\bD_0=\bx$ coincides with the
conditional distribution of $\bvartheta_{k+1,h}^{\rm LMCO'}$ given $\bvartheta_{k,h}^{\rm LMCO'}=\bx$.
Therefore, if $\bD_0\sim \nu'_k$, then $\bD_h\sim \nu'_{k+1}$ and, consequently,
$W_2(\nu'_{k+1},\pi)^2 \le \bfE[\|\bD_h-\bL_h\|_2^2]$.

To ease notation, we set $\bfH_0 = \nabla^2 f(\bD_0)$. The process $\bD$ satisfies the SDE
\begin{align}\label{SDE2}
d\bD_t &= -\big[(\bfI_p-t\nabla^2 f(\bD_0))\nabla f(\bD_0) + \sqrt{2}\,\bfH_0\bW_t\big]\,dt+\sqrt{2}\,d\bW_t,
\end{align}
which implies that
\begin{align}\label{SDE3}
d\bD_t = &-\big[\nabla f(\bD_0) + \nabla^2 f(\bD_0)(\bD_t-\bD_0)\big]\,dt +\sqrt{2}\,d\bW_t\\
&-0.5 t^2\bfH_0^2\nabla f(\bD_0)\,dt - \sqrt{2}\,\bfH_0^2\int_0^t (t-u)\,d\bW_{u}\,dt.
\end{align}
Proceeding in the same way as for getting \eqref{fourint}, we arrive at the decomposition
$\bX_t = \bDelta_t-\bDelta_0 = A_t-B_t-C_t+S_t - E_t-F_t$,
where $A_t$, $B_t$, $C_t$ and $S_t$ stand for the four integrals in \eqref{fourint} whereas $E_t$ and $F_t$
are
\begin{align}
E_t &= 0.5\int_0^t  e^{-s\bfH_0}  s^2\, ds\,\bfH_0^2\nabla f(\bD_0)\\
F_t &= \sqrt{2}\,\bfH_0^2\int_0^t e^{-s\bfH_0} \int_0^s (s-u) \,d\bW_u\, ds.
\end{align}
Using the properties of the stochastic integral, we get
\begin{align}
\bfE[\|F_h\|_2^2]
		&= 2\bfE\Big[\Big\|\bfH_0^2\int_0^h e^{-s\bfH_0} \int_0^s (s-u) \,d\bW_u\, ds\Big\|_2^2\Big]\\
		&= 2\bfE\Big[\Big\|\int_0^h \int_u^h \bfH_0^2e^{-s\bfH_0} (s-u)\,ds\,d\bW_u\Big\|_2^2\Big]\\
		&= 2\int_0^h\Big\|\int_u^h \bfH_0^2e^{-s\bfH_0} (s-u)\,ds\Big\|_F^2\,du\\
		&\le 2M^4p \int_0^h\Big(\int_u^h (s-u)\,ds\Big)^2\,du = \frac{M^4h^5p}{10}.\label{Ft}
\end{align}
On the other hand,
\begin{align}
\|E_h\|_2
		&\le 0.5M^2\int_0^hs^2\,ds\|\nabla f(\bD_0)\|_2
		\le \frac{M^2h^3}{6}\big(\|\nabla f(\bL_0)\|_2+M\|\bDelta_0\|_2\big),
\end{align}
which, in view of \Cref{lemB}, implies that
\begin{align}
\|E_h\|_{L_2}^2
		\le \frac{M^2h^3}{6}\big(\sqrt{Mp}+M W_2(\nu'_k,\pi)\big).\label{Et}
\end{align}
Proceeding as in \eqref{eq12'} and using \eqref{eq15}, we get
\begin{align}
{\|\bDelta_h\|}_{L_2}
		& = \|\bDelta_0+A_h-B_h-C_h+S_h-E_h-F_h\|_{L_2} \\
		&\le \|\bDelta_0+A_h+S_h-F_h\|_{L_2} + \|B_h\|_{L_2} + \|C_h\|_{L_2}  +\|E_h\|_{L_2}\\
		&\le (\|\bDelta_0+A_h\|_{L_2}^2 +\|S_h-F_h\|_{L_2}^2)^{1/2} + \|B_h\|_{L_2} + \|C_h\|_{L_2}
				+ \|E_h\|_{L_2}.\label{decomp3}
\end{align}
Using the last but one estimate in \eqref{St}, in conjunction with \eqref{Ft}, we get
inequalities
\begin{align}
\|S_h\|_{L_2}^2 &\le (\nicefrac23)M_2Mh^3p W_2(\nu_k',\pi)\\
|\bfE[S_h^\top F_h]| &\le (\nicefrac1{\sqrt{15}})M^2 M_2 h^4pW_2(\nu_k',\pi),
\end{align}
which, for $h\le 3m/(4M^2)$, imply that $\|S_h-F_h\|_{L_2}^2$ is less than or equal to
\begin{align}
(\nicefrac23)M_2Mh^3p W_2(\nu_k',\pi) &+ (\nicefrac2{\sqrt{15}})M^2 M_2 h^4pW_2(\nu_k',\pi)+(\nicefrac1{10})M^4h^5p\\
		&\le 1.06 M_2Mh^3p W_2(\nu_k',\pi) + 0.1 M^4 h^5p.
\end{align}
Injecting this bound, \eqref{At}, \eqref{Bt}, \eqref{Ct} and \eqref{Et} in \eqref{decomp3}, we arrive at
\begin{align}
{\|\bDelta_h\|}_{L_2}
		&\le \big\{\big[(1-mh+0.5M^2h^2)^2W_2(\nu_k',\pi)^2 + 1.06M_2Mh^3p W_2(\nu_k',\pi)  + 0.1 M^4 h^5p\big\}^{1/2} \\
		&\quad+ 0.88M_2h^2(p+1) + \Big(\mu + \frac{M^3h^3}{6}\Big)W_2(\nu_k',\pi) +\frac{M^2M_2h^4(p+1)}{16\mu}
				+ \frac{M^{5/2}h^3\sqrt{p}}{6}.\label{decomp3'}
\end{align}
In view of the inequality $\sqrt{a^2+b+c}\le \sqrt{a^2+c}+(\nicefrac{b}{2a})$, the last display leads to
\begin{align}
W_2(\nu_{k+1}',\pi)
		&\le \big\{\big[(1-mh+0.5M^2h^2)^2W_2(\nu_k',\pi)^2 +  0.1 M^4 h^5p\big\}^{1/2} \\
		&\quad+ \frac{0.53M_2Mh^3p}{1-mh+0.5M^2h^2} + 0.88M_2h^2(p+1) + \Big(\mu + \frac{M^3h^3}{6}\Big)W_2(\nu_k',\pi)\\
		&\quad+ \frac{M^2M_2h^4(p+1)}{16\mu}
				+ \frac{M^{5/2}h^3\sqrt{p}}{6}.\label{decomp4}
\end{align}
For $h\le 3m/(4M^2)$ and  $\mu = 0.25mh$, we can use the inequality $1-mh+0.5M^2h^2 \ge 17/32$ and
simplify the last display as follows:
\begin{align}
W_2(\nu_{k+1}',\pi)
		&\le \big\{\big[(1-mh+0.5M^2h^2)^2W_2(\nu_k',\pi)^2 +  0.1 M^4 h^5p\big\}^{1/2} \\
		&\quad+ \frac{0.3975 M_2h^2(p+1)}{1-mh+0.5M^2h^2} + 0.88M_2h^2(p+1) + \Big(\mu + \frac{M^3h^3}{6}\Big)W_2(\nu_k',\pi)\\
		&\quad+ \frac{3M_2h^2(p+1)}{16}
				+ \frac{M^{5/2}h^3\sqrt{p}}{6}\\
		&\le \big\{(1-mh+0.5M^2h^2)^2W_2(\nu_k',\pi)^2 +  0.1 M^4 h^5p\big\}^{1/2} \\
		&\quad + \Big(0.25mh + \frac{M^3h^3}{6}\Big)W_2(\nu_k',\pi) + 1.82M_2h^2(p+1) + \frac{M^{5/2}h^3\sqrt{p}}{6}.
\end{align}
We apply \Cref{lemI} to the sequence $x_k = W_2(\nu_k',\pi)$ with $A = mh-0.5M^2h^2$ and
$D = 0.25mh + M^3h^3/6$. For $h\le 3m/(4M^2)$ we have $A-D = 0.75mh- 0.5M^2h^2- (Mh)^3/6 \ge 0.25mh$ and
$A+D \le 1.25 mh -(3/8)M^2h^2 \le	0.727$. This yields
\begin{align}
W_2(\nu_{k+1}',\pi)
		&\le (1-0.25 mh)^k W_2(\nu'_0,\pi) + \frac{7.28 M_2h(p+1)}{m}  + \frac{2M^{5/2}h^2\sqrt{p}}{3m}+
		\frac{2\sqrt{0.1}\,M^2h^2\sqrt{p}}{\sqrt{1.273m}}\\
		&\le (1-0.25 mh)^k W_2(\nu'_0,\pi) + \frac{7.28 M_2h(p+1)}{m}  + \frac{1.23M^{5/2}h^2\sqrt{p}}{m}.
\end{align}
This completes the proof of \eqref{A3''} and that of the theorem.

\begin{proof}[Proof of \Cref{propB}]
Let us denote $\bfM_k =\int_0^he^{-s\bfH_k}\,ds\int_0^1 \nabla^2 f(\bD_{kh}+x\bDelta_{k})\,dx$.
From \eqref{fourint}, we have $\bDelta_{k+1}=\bDelta_k + A_{k,h}  + G_{k,h}$
with
\begin{align}
A_{k,h} &= \int_0^he^{-s\bfH_k}\,ds\big(\nabla f(\bD_{kh})-\nabla f(\bL_{kh})\big)
			=  -\bfM_k\bDelta_k,\\
G_{k,h} &= \int_0^h e^{-s\bfH_k} \big(\nabla f(\bL_{kh})-\nabla f(\bL_s)+\bfH_k(\bL_s-\bL_{kh})\big)\,ds.
\end{align}
Using the fact that
\begin{align}
\bigg\|\int_0^1 \nabla^2 f(\bD_{kh}+x\bDelta_{k})\,dx - \bfH_k\bigg\|
		& \le  \int_0^1 \big\|\nabla^2 f(\bD_{kh}+x\bDelta_{k}) - \bfH_k\big\|\,dx\le \frac{M_2}2\,\|\bDelta_k\|_2,
\end{align}
we get $\|\bDelta_k+A_{k,h}\|_2= \|(\bfI-\bfM_k)\bDelta_k\|_2\le \frac{M_2}{2m}\,\|\bDelta_k\|_2^2+e^{-mh}\|\bDelta_k\|_2$.
This further leads to the recursive inequality
\begin{align}
\|\bDelta_{k+1}\|_2 &\le \frac{M_2}{2m}\,\|\bDelta_k\|_2^2+e^{-mh}\|\bDelta_k\|_2 + \|G_{k,h}\|_2.
\end{align}
In view of the Minkowski inequality, this yields
\begin{align}\label{rec2}
(\bfE[\|\bDelta_{k+1}\|_2^q])^{1/q} &\le \frac{M_2}{2m}\,\bfE[\|\bDelta_k\|_2^{2q}]^{1/q}+e^{-mh}
\bfE[\|\bDelta_k\|_2^{2q}]^{1/2q} + \bfE[\|G_{k,h}\|_2^q]^{1/q}.
\end{align}
We choose some $K\in\NN$ and define the sequence $\{x_0,\ldots,x_K\}$ by setting
$x_k^{2^{K+1-k}} = \bfE[\|\bDelta_{k}\|_2^{2^{K+1-k}}]$. Choosing in \eqref{rec2} $q = 2^{K-k}$,
we get
\begin{align}\label{rec3}
x_{k+1} &\le \frac{M_2}{2m}\,x_k^2+e^{-mh}
x_k + \bfE[\|G_{k,h}\|_2^{2^{K-k}}]^{2^{k-K}},\quad k=0,1,\ldots,K-1.
\end{align}
We are in a position to apply
\Cref{lemH} to the sequence $\{x_k\}_{k=0,\ldots,K}$. This yields
\begin{align}
x_{K}
		&\le \frac{2m}{M_2}\bigg(\frac{M_2x_0}{2m}+\frac12e^{-mh}\bigg)^{2^K}
		\exp\bigg\{2^{K}
		\frac{M_2\max_{k} \bfE[\|G_{k,h}\|_2^{2^K}]^{2^{-K}} + me^{-mh}}
		{m(\frac{M_2x_0}{2m}+\frac12e^{-mh})^{2^{K}}}\bigg\},\label{xK}
\end{align}
where $\max_k$ is a short notation for $\max_{k=0,1,\ldots,K-1}$.
It suffices now to upper bound the moments of $\|G_{k,h}\|_2$. We have
\begin{align}
\bfE[\|G_{k,h}\|_2^{q}]^{1/q}
		&\le M\int_0^h e^{-sm} \big(\bfE[\|\bL_{kh+s}-\bL_{kh}\|_2^{q}]\big)^{1/q}\,ds\\
		&\le M\int_0^h e^{-sm} \Big\{\big(\bfE[\|\int_0^s\nabla f(\bL_{kh+u})\,du\|_2^{q}]\big)^{1/q}
				+\sqrt2\big(\bfE[\|\bW_s\|_2^{q}]\big)^{1/q}\Big\}\,ds\\
		&\le M\int_0^h e^{-sm} s\,ds\big(\bfE[\|\nabla f(\bL_0)\|_2^{q}]\big)^{1/q} +M\sqrt{2p+q-2}
	  \int_0^se^{-sm}\sqrt{s}\,ds\\
		&\le \frac{M}{m^2}\big(\bfE[\|\nabla f(\bL_0)\|_2^{q}]\big)^{1/q} + \frac{M}{2m^{3/2}}\sqrt{(2p+q-2)\pi}.
\end{align}
On the other hand, by integration by parts, for every $q\in2\NN$, we have
\begin{align}
\bfE[\|\nabla f(\bL_0)\|_2^{q}]
	&= -\int_{\RR^p} \|\nabla f(\bx)\|_2^{q-2}\,\nabla f(\bx)\!^\top d \pi(\bx)\\
	&= \sum_{\ell=1}^p \int_{\RR^p} \partial_\ell \Big(\|\nabla f(\bx)\|_2^{q-2}\,\partial_\ell f(\bx)\Big)  \pi(\bx)\,d\bx\\
	&\le M(p+q-2)\bfE[\|\nabla f(\bL_0)\|_2^{q-2}].
\end{align}
This yields $(\bfE[\|\nabla f(\bL_0)\|_2^{q}])^{1/q}\le \sqrt{M(p+0.5q-1)}$. Combining all these estimates, we arrive at
$$
\bfE[\|G_{k,h}\|_2^{q}]^{1/q} \le  \frac{1.6M^{3/2}\sqrt{2p+q-2}}{m^2}.
$$
Combining this inequality with \eqref{xK} and replacing $x_K$ by $(\bfE[\|\bDelta_K\|_2^2])^{1/2}$, we get
\begin{align}
(\bfE[\|\bDelta_K\|_2^2])^{1/2}
		&\le \frac{2m}{M_2}\bigg(\frac{M_2x_0}{2m}+\frac12e^{-mh}\bigg)^{2^K}
		\exp\bigg\{2^{K}
		\frac{1.6 M_2 M^{3/2}\sqrt{2p+2^{K-1}-2} + m^3e^{-mh}}
		{m^3(\frac{M_2x_0}{2m}+\frac12e^{-mh})^{2^{K}}}\bigg\}.
\end{align}
This completes the proof of the proposition.
\end{proof}

\subsection{Proofs of lemmas}\label{ssecLem}

Here we provide the proofs of \Cref{lemA}, \Cref{lemB} and \Cref{lemC}.
\begin{proof}[Proof of \Cref{lemA}]
We start by recalling the following inequality  \citep[Theorem 2.12]{Nest}, true for any $m$-strongly convex and $M$-gradient Lipschitz function $f$:
\begin{equation}\label{thnest}
    (\by-\bx)^\top\left(\nabla f(\by)- \nabla f(\bx)\right) \geq \frac{mM}{m+M}\|\by-\bx\|^2_2 +\frac{1}{m+M}\left\|\nabla f(\by) - \nabla f(\bx)\right\|_2^2,
\end{equation}
for all vectors $\bx$ and $y$ from $\RR^p$.  This yields
\begin{align}
    \|\by-\bx-h(\nabla f(\by) - &\nabla f(\bx))\|_2^2 \\
    &= \|\by-\bx\|^2_2 - 2h(\by-\bx)^\top(\nabla f(\by)  - \nabla f(\bx)) + h^2 \|\nabla f(\by) - \nabla f(\bx)\|_2^2\\
    &\leq \left(1 - \frac{2hmM}{m+M}\right)\|\by-\bx\|_2^2 +
    h\left(h - \frac{2}{m+M}\right)\|\nabla f(\by) - \nabla f(\bx)\|_2^2.
\end{align}
Since $f$ is $m$-strongly convex,  we have (\cite{Nest}, Theorem 2.1.9)
\begin{equation}
    \|\nabla f(\by) - \nabla f(\bx)\|_2 \geq m\|\by-\bx\|_2.
\end{equation}
In the case $h \le \frac{2}{m+M}$, applying  the previous result to the second summand, we get
\begin{equation}
      \|\by-\bx-h(\nabla f(\by) - \nabla f(\bx))\|_2^2 \leq (1-hm)^2 \|\by-\bx\|^2.
\end{equation}
In the case when $h \geq \frac{2}{m+M}$, we
use the Lipschitz continuity of $\nabla f$, which leads to
\begin{equation}
      \|\by-\bx-h(\nabla f(\by) - \nabla f(\bx))\|_2^2 \leq (hM - 1)^2 \|\by-\bx\|^2.
\end{equation}
Summing up, for all $h\in(0,2/M)$ we have shown
\begin{equation}
      \|\by-\bx-h(\nabla f(\by) - \nabla f(\bx))\|_2^2 \leq \left\{(1-hm)^2 \vee(hM - 1)^2\right\} \|\by-\bx\|^2.
\end{equation}
This completes the proof.
\end{proof}
\begin{proof}[Proof of \Cref{lemB}]
We start the proof with the case $p=1$. The function
$x\mapsto f'(x)$ being Lipschitz continuous is almost surely differentiable.
Furthermore, it is clear that $|f''(x)|\le M$  for every $x$ for which
this second derivative exists. The result of \citep[Theorem 7.20]{rudin87}
implies that
\begin{equation}
f'(x)-f'(0) = \int_0^x f''(y)\,dy.
\end{equation}
Therefore, using the relation  $f'(x)\,\pi(x) = -\pi'(x)$, we get
\begin{align}
\int_{\RR} f'(x)^2\,\pi(x)\,dx
		& = f'(0)\int_{\RR} f'(x)\,\pi(x)\,dx +
				\int_{\RR}\Big(\int_0^x f''(y)\,dy\Big) f'(x)\,\pi(x)\,dx \\
		& = -f'(0)\int_{\RR} \pi'(x)\,dx -
				\int_{\RR}\Big(\int_0^x f''(y)\,dy\Big) \pi'(x)\,dx \\
		& = -\int_{0}^\infty\int_0^x f''(y)\,\pi'(x)\,dy\,dx
				+\int_{-\infty}^0\int_x^0 f''(y)\,\pi'(x)\,dy\,dx.
\end{align}
In view of Fubini's theorem, we arrive at
\begin{align}\label{one-dim}
\int_{\RR} f'(x)^2\,\pi(x)\,dx & =
		\int_{0}^\infty f''(y)\,\pi(y)\,dy
				+\int_{-\infty}^0 f''(y)\,\pi(y)\,dy\le M.
\end{align}
Now let us return to the multidimensional case:
\begin{equation}
    \int_{\RR^p} \|\nabla f(\bx)\|_2^2\,\pi(\bx)\,d\bx =
    \sum\limits_{k=1}^{p}\int_{\RR^p} \left(\frac{\partial f}{\partial x_k} (\bx)\right)^2\pi(\bx)\,d\bx.
\end{equation}
We will show that each of the summands is less than $M$, thus the sum is less than
$Mp$. Let us prove it for $k=1$. The proof is similar for the case $k>1$.
Using Fubini's theorem, we have
\begin{equation}
    \int_{\RR^p} \left(\frac{\partial f}{\partial x_1} (\bx)\right)^2\pi(\bx)\,d\bx
    = \int_{\RR}\ldots \int_{\RR} \left(\frac{\partial f}{\partial x_1} (x_1,x_2,\ldots,x_p)\right)^2\pi(x_1,x_2,\ldots,x_p)\,dx_1dx_2\ldots dx_p.
\end{equation}
Let us fix the $(p-1)$-tuple $(x_2,x_3,\ldots,x_p)$ and define functions $g$ and $\eta$
as $g(t) = f(t,x_2,\ldots,x_p)$ and
$\eta(t) = \pi(t,x_2,\ldots,x_p)$, respectively.
It is easy to verify that $\eta$ is an integrable
log-concave function, with $g$ as its potential. The latter is also differentiable and its
derivative is Lipschitz-continuous with constant $M$. Thus we have
\begin{equation}
     \int_{\RR} \left(\frac{\partial f}{\partial x_1} (x_1,x_2,\ldots,x_p)\right)^2\pi(x_1,x_2,\ldots,x_p)\,dx_1
     =  \int_{\RR} \left(g'(t)\right)^2\eta(t)dt.
\end{equation}
From the definition one can verify that $\int_{\RR}\eta(t)dt = \pi_1(x_2,\ldots,x_p)$, where $\pi_1$ is the marginal distribution of all the coordinates except the first. Therefore,
\begin{align}
    \int_{\RR} g'(t)^2\eta(t)dt &= \pi_1(x_2,\ldots,x_p)
    \int_{\RR} g'(t)^2\frac{\eta(t)}{\pi_1(x_2,\ldots,x_p)}dt\\
    &\leq M\pi_1(x_2,\ldots,x_p)
\end{align}
The last inequality is true due to \eqref{one-dim}. Returning
to our initial integral, we obtain
\begin{equation}
    \int_{\RR^p} \left(\frac{\partial f}{\partial x_1} (\bx)\right)^2\pi(\bx)\,d\bx
    \leq M\int_{\RR^{p-1}}\pi_1(x_2,\ldots,x_p) dx_2\ldots dx_p = M.
\end{equation}
This completes the proof.
\end{proof}

\begin{proof}[Proof of \Cref{lemC}]
Since the process $\bL$ is stationary, $V(a)$ has the same distribution as $V(0)$. For this
reason, it suffices to prove the claim of the lemma for $a=0$ only.
Using the Cauchy-Schwarz inequality and the Lipschitz continuity of $f$, we get
\begin{align}
\|\bV(0)\|_{L_2}
		& = \Big\|\int_{0}^{h}\big(\nabla f(\bL_t) - \nabla f(\bL_{0})\big)\,dt\Big\|_{L_2}\\
		& \le  \int_{0}^{h}\big\|\nabla f(\bL_t) - \nabla f(\bL_{0})\big\|_{L_2}\,dt\\
		& \le  M\int_{0}^{h}\big\|\bL_t - \bL_{0}\big\|_{L_2}\,dt.
\end{align}		
Combining this inequality with the definition of $\bL_t$, we arrive at
\begin{align}		
\|\bV(0)\|_{L_2}
		& \le M\int_{0}^{h}\big\|-\int_{0}^t \nabla f(\bL_s)\,ds + \sqrt{2}\,\bW_{t}\big\|_{L_2}\,dt\\
		& \le M\int_{0}^{h}\big\|\int_{0}^t \nabla f(\bL_s)\,ds\big\|_{L_2}\,dt
				+ M\int_{0}^{h}\big\|\sqrt{2}\,\bW_{t}\big\|_{L_2}\,dt\\
		& \le M\int_{0}^{h}\int_{0}^t \|\nabla f(\bL_s)\|_{L_2}\,ds\,dt
				+ M\int_{0}^{h}\sqrt{2pt} \,dt.
\end{align}
In view of the stationarity of $\bL_t$, we have $\|\nabla f(\bL_s)\|_{L_2} = \|\nabla f(\bL_0)\|_{L_2}$,
which leads to
\begin{align}		
\|\bV(0)\|_{L_2}
		& \le (\nicefrac{1}{2}) M h^2 \big\|\nabla f(\bL_0)\big\|_{L_2} + (\nicefrac{2}{3}) M\sqrt{2p}\;  h^{3/2}.
\end{align}
To complete the proof, it suffices to apply \Cref{lemB}.
\end{proof}

\begin{lemma}\label{lemF}
Let us denote
\begin{align}
\tilde\bV &=\int_0^h\big(\nabla f(\bL_t) - \nabla f(\bL_{0})-\nabla^2 f(\bL_0)(\bL_t-\bL_0)\big)\,dt,\\
\bar\bV &=\int_0^h\Big\{\nabla f(\bL_t) - \nabla f(\bL_{0})-\sqrt2\,\int_0^t\nabla^2 f(\bL_s)d\bW_s\Big\}\,dt,
\end{align}
with $f$ satisfying {\bf Condition F} and $h\le 1/M$, then
\begin{align}
(\bfE[\|\tilde\bV\|_2^2])^{1/2} &\le 0.877 M_2h^2(p^2+2p)^{1/2},\label{eqB}\\
\|\bar\bV\|_{L_2} &\le (\nicefrac12)(M^{3/2}\sqrt{p}+M_2p) h^2.		
\label{eqC}
\end{align}
\end{lemma}

\begin{proof}
We first note that we have
\begin{align}
\|\tilde\bV\|_2
		&\le \int_0^h \|\int_0^1\big(\nabla^2\!f(\bL_0+x(\bL_t-\bL_0))-\nabla^2\!f(\bL_0)\big)\,dx(\bL_t-\bL_0)\|_2\,dt\\
		&\le 0.5M_2\int_0^h \|\bL_t-\bL_0\|_2^{2}\,dt.
\end{align}
In view of \eqref{eqA}, this implies that
$(\bfE[\|\tilde\bV\|_2^2])^{1/2}\le 0.5M_2\int_0^h (\bfE[\|\bL_t-\bL_0\|_2^{4}])^{1/2}\,dt$.
Using the triangle inequality
and integration by parts (precise details of the computations are omitted in the
interest of saving space), we arrive at
\begin{align}
\bfE[\|\bL_t-\bL_0\|_2^{4}]
		&\le \bfE[\|\int_0^t\nabla f(\bL_s)\|_2^4] +4\bfE[\|\bW_t\|_2^4]\\
		&\qquad	
		+ 12\bigg(\bfE[\|\int_0^t\nabla f(\bL_s)\|_2^4]\bfE[\|\sqrt{2}\bW_t\|_2^4]\bigg)^{1/2}\\
		&\le t^4 M^2 p(2+p) + 12 t^3 Mp(2+p) + 4t^2p(2+p)\\
		& = p(2+p)t^2(t^2M^2+12tM+4).\label{eq14}
\end{align}
Integrating this inequality, we get
\begin{align}
(\bfE[\|\tilde\bV\|_2^2])^{1/2}
		&\le 0.5M_2(p^2+2p)^{1/2}\int_0^h t(t^2M^2+12tM+4)^{1/2}\,dt\\
		&\le \frac{0.5M_2(p^2+2p)^{1/2}}{M^2}\int_0^{Mh} t(t^2+12t+4)^{1/2}\,dt\\
		&\le {0.5M_2h^2(p^2+2p)^{1/2}}\sup_{x\in(0,2]}\frac1{x^2}\int_0^{x} t(t^2+12t+4)^{1/2}\,dt\\
		&= \frac{0.5M_2h^2(p^2+2p)^{1/2}}{4}\int_0^{2} t(t^2+12t+4)^{1/2}\,dt\\
		&\le 1.16M_2h^2(p^2+2p)^{1/2}.
\end{align}
This completes the proof of \eqref{eqB}. To prove \eqref{eqC}, we first assume that
$f$ is three times continuously differentiable and apply the Ito formula:
$$
\nabla f(\bL_t) -\nabla f(\bL_0) = \int_0^t \nabla^2 f(\bL_s)\, d\bL_s +
\int_0^t \Delta [\nabla f(\bL_s)]\,ds.
$$
Let us check that $\|\Delta [\nabla f(\bx)]\|_{2}=\|\nabla[\Delta f(\bx)]\|_{2}
\le M_2p$ for every $\bx\in\RR^p$. Indeed, let
us introduce the function $g:\RR^p\to\RR$ defined by $g(\bx) =\Delta f(\bx) =
\tr[\nabla^2 f(\bx)]$. The third item of condition F implies that $|g(\bx+t\bu)
-g(\bx)|\le p M_2|t|$ for every $t\in\RR$ and every unit vector $\bu\in\RR^p$.
Therefore, letting $t$ go to zero, we get $|\bu^\top \nabla g(\bx)|\le pM_2$ for
every unit vector $\bu$. Choosing $\bu$ proportional to $\nabla g(\bx)$, we get
the inequality $\|\nabla g(\bx)\|_2 = \|\nabla [\Delta f(\bx)]\|_2\le pM_2$.
This leads to
\begin{align}
\|\bar\bV\|_{L_2}
		&\le \int_0^h\int_0^t
		\big\|\nabla^2 f(\bL_s)\nabla f(\bL_s) - \Delta [\nabla f(\bL_s)]\big\|_{L^2}\,ds\,dt\\
		&\le \int_0^h\int_0^t \big(M\big\|\nabla f(\bL_s)\big\|_{L^2} + M_2p\big)\,ds\,dt\\
		&=(\nicefrac12)(M^{3/2}\sqrt{p}+M_2p) h^2.
\end{align}
This completes the proof of the lemma in the case of three times continuously differentiable functions $f$.
If $f$ is two-times differentiable with a second-order derivative satisfying the Lipschitz condition, then
we can choose an arbitrarily small $\delta>0$ and apply the previous result to the smoothed function
$f_\delta = f*\varphi_{\delta}$. Here, $\varphi_\delta$ denotes the density of the Gaussian distribution
$\mathcal N_p(0,\delta^2\bfI_p)$ and ``$*$'' is the convolution operator. The formula $\nabla^2 f_\delta =
(\nabla^2 f)*\varphi_\delta$ implies that $f_\delta$ satisfies the required smoothness assumptions with
the same constants $M$ and $M_2$ as the function $f$. Thus, defining $\bar\bV_\delta$ in the same way
as $\bar\bV$ with $f_\delta$ instead of $f$, we get
\begin{align}
\|\bar\bV_\delta\|_{L_2} &\le (\nicefrac12)(M^{3/2}\sqrt{p}+M_2p) h^2.\label{bvdelta1}
\end{align}
On the other hand, setting $g_\delta = f-f_\delta$, we get
\begin{align}
\|\bar\bV_\delta - \bar\bV\|_{L_2}
		&\le \int_0^h\Big\|\nabla g_\delta(\bL_t) - \nabla g_\delta(\bL_{0})-\sqrt2\,\int_0^t\nabla^2 g_\delta(\bL_s)d\bW_s\Big\|_{L^2}\,dt\\
		&\le \int_0^h\big\|\nabla g_\delta(\bL_t) - \nabla g_\delta(\bL_{0})\big\|_{L^2}\,dt\\
		&\qquad +
		\sqrt{2p}\,\int_0^h\bigg(\int_0^t\bfE\|\nabla^2 g_\delta(\bL_s)\|^2ds\bigg)^{1/2}\,dt.
\end{align}
Using the Lipschitz continuity of $\nabla f$ and $\nabla^2 f$, one easily checks that
\begin{align}
\|\nabla g_\delta(\bx)\|_2
		&\le \int_{\RR^p} \|\nabla f(\bx-\by)-\nabla f(\bx)\|_2 \varphi_\delta(\by)\,d\by \\
		&\le M \int_{\RR^p}\|\by\|_2\varphi_\delta(\by)\,d\by\le M\delta\sqrt{p},\\
\|\nabla^2 g_\delta(\bx)\|
		&\le \int_{\RR^p} \|\nabla^2 f(\bx-\by)-\nabla^2 f(\bx)\| \varphi_\delta(\by)\,d\by\\
		&\le M_2 \int_{\RR^p}\|\by\|_2\varphi_\delta(\by)\,d\by\le M_2\delta\sqrt{p}.
\end{align}
This implies that the limit, when $\delta$ tends to zero, of $\|\bar\bV_\delta - \bar\bV\|_{L_2}$ is equal to zero.
As a consequence,
\begin{align}
\|\bar\bV\|_{L_2}
	&\le  \lim_{\delta\to 0}  \big(\|\bar\bV_\delta\|_{L_2}+\|\bar\bV_\delta - \bar\bV\|_{L_2}\big)\\
	&\le (\nicefrac12)(M^{3/2}\sqrt{p}+M_2p) h^2+\lim_{\delta\to 0}  \|\bar\bV_\delta - \bar\bV\|_{L_2}\\
	&\le (\nicefrac12)(M^{3/2}\sqrt{p}+M_2p) h^2.
\end{align}
This completes the proof of the lemma.
\end{proof}

\begin{lemma}\label{lemE}
Let $A$, $B$ and $C$ be given non-negative numbers such that $A\in (0,1)$. Assume that the sequence of non-negative numbers
$\{x_k\}_{k\in \NN}$ satisfies the recursive inequality
\begin{align}
x_{k+1}^2&\le [(1-A)x_k+C]^2+B^2
\end{align}
for every integer $k\ge 0$. Let us denote
\begin{align}
E&= \frac{(1-A)C+\big\{C^2+(2A-A^2)B^2\big\}^{1/2}}{2A-A^2}  \ge \frac{(1-A)C}{A(2-A)}+\frac{B}{\sqrt{A(2-A)}}\\
D&= \big\{[(1-A)E+C]^2+B^2\big\}^{1/2} - (1-A)E\le C + \frac{B^2A}{C+\sqrt{A(2-A)}\,B}
\end{align}
Then
\begin{align}
x_k &\le (1-A)^{k} x_0 + \frac{D}{A} \le (1-A)^{k} x_0 + \frac{C}{A} + \frac{B^2}{C+\sqrt{A(2-A)}\,B}\label{inq3}
\end{align}
for all integers $k\ge 0$.
\end{lemma}

\begin{proof} We will repeatedly use the fact that $D = EA$.
Let us introduce the sequence $y_k$ defined as follows: $y_0 = x_0 + E$
and
\begin{align}
y_{k+1} = (1-A) y_k + D, \quad k=0,1,2,\ldots
\end{align}
We will first show that $y_k\ge x_k\vee E$ for every $k\ge 0$. This can be done by
mathematical induction. For $k=0$, this claim directly follows from the definition of $y_0$.
Assume that for some $k$, we have $x_k\le y_k$ and $y_k\ge E$. Then, for $k+1$, we
have
\begin{align}
x_{k+1}
		&\le \big([(1-A)x_k+C]^2+B^2\big)^{1/2} \\
		&\le \big([(1-A)y_k+C]^2+B^2\big)^{1/2} \\
		& = (1-A)y_k + \big([(1-A)y_k+C]^2+B^2\big)^{1/2} - (1-A)y_k\\
		&\le (1-A)y_k + \big([(1-A)E+C]^2+B^2\big)^{1/2} - (1-A)E
		 = y_{k+1}
\end{align}
and, since $D= EA$, $y_{k+1} = (1-A) y_k + D \ge (1-A)E  +  EA = E$.
Thus, we have checked that the sequence $x_k$ is dominated by the sequence $y_k$. It remains to
establish an upper bound on $y_k$. This is an easy task since $y_k$ satisfies a first-order
linear recurrence relation. We get
\begin{align}
y_{k} &= (1-A)^{k-1} y_1 + \sum_{j=0}^{k-2}(1-A)^j D\\
			&= (1-A)^{k-1} \Big(x_1 +\frac{D}{A}\Big) + \frac{D}{A}\big(1-(1-A)^{k-1}\big)\\
			&= (1-A)^{k-1} x_1 +\frac{D}{A}.
\end{align}
This completes the proof of \eqref{inq3}.
\end{proof}

\begin{proof}[Proof of \Cref{lemG}]
Let us introduce the $\RR^p$-valued random process
$\bv_t = -\exp\big\{\int_0^t\bfH_u\,du\big\}\int_0^t \bfH_s\bx_s\,ds$.
The time derivative of this process satisfies
$$
\bv'_t = - \exp\Big\{\int_0^t\bfH_u\,du\Big\}\bfH_t\balpha_t.
$$
This implies that $\bv_t = - \int_0^t\exp\big\{\int_0^s\bfH_u\,du\big\}\bfH_s\balpha_s\,ds$. Using the definition of
$\bv_t$,  we can check that
$\int_0^t \bfH_s\bx_s\,ds = -\exp\big\{-\int_0^t\bfH_u\,du\big\}\bv_t =
\int_0^t\exp\big\{-\int_s^t\bfH_u\,du\big\}\bfH_s\balpha_s\,ds$. Substituting this in \eqref{recx}, we get
\begin{align}\label{intermedx}
\bx_t
		&= \balpha_t -\int_0^t\exp\big\{-\int_s^t\bfH_u\,du\big\}\bfH_s\balpha_s\,ds.
\end{align}
On the other hand---using the notation $\bfM_t = \exp\big\{\int_0^t\bfH_u\,du\big\}$ and
the integration by parts formula for semi-martingales---the second integral on the right hand side of \eqref{finalx} can be
modified as follows:
\begin{align}
\int_0^t \exp\Big\{-\int_s^t\bfH_u\,du\Big\}d\balpha_s
		& = \bfM_t^{-1}\int_0^t \bfM_sd\balpha_s\\
		& = \bfM_t^{-1}\Big(\bfM_t\balpha_t-\bfM_0\balpha_0-\int_0^t d\bfM_s\,\balpha_s\Big)\\
		&= \balpha_t - \exp\Big\{-\int_0^t\bfH_u\,du\Big\}\balpha_0 \\
		&\qquad	- \int_0^t \exp\Big\{-\int_s^t\bfH_u\,du\Big\}\bfH_s\balpha_s\,ds.
\end{align}
Combining this equation with \eqref{intermedx}, we get the claim of the lemma.
\end{proof}

\begin{lemma}\label{lemH} Let $A$ and $B$ be given positive numbers and $\{C_k\}_{k\in\NN}$ be a given sequence
of real numbers. Assume that the sequence $\{x_k\}_{k\in\NN}$ satisfies the recursive inequality
\begin{align}
x_{k+1} \le Ax_k^2+2Bx_k+C_k,\qquad \forall k\in\NN.
\end{align}
Then, for all $k\in\NN$,
\begin{align}
x_k \le \frac1A \big(Ax_0+B\big)^{2^k}\exp\bigg\{\sum_{j=0}^{k-1} 2^{k-1-j}\,
		\frac{AC_j+B(1-B)}{(Ax_0+B)^{2^{j+1}}}\bigg\}.
\end{align}
\end{lemma}

\begin{proof}
Let us introduce the sequences $\{y_k\}_{k\in\NN}$ and $\{z_k\}_{k\in\NN}$
defined by the relations $y_0=x_0$,
\begin{align}
y_{k+1} &= Ay_k^2+2By_k+C_k,\\
z_k     &= (Ax_0+B)^{2^k}\exp\bigg\{\sum_{j=0}^{k-1} 2^{k-1-j}\,
		\frac{AC_j+B(1-B)}{(Ax_0+B)^{2^{j+1}}}\bigg\}.
\end{align}
Using mathematical induction, one easily shows that inequalities
\begin{align}
x_k\le y_k\qquad \text{and}\qquad  (Ax_0+B)^{2^k}\le Ay_k+B \le z_k
\end{align}
hold for every $k\in\NN$. As a consequence, we get
\begin{align}
x_{k}\le \frac{Ax_k+B}{A}\le \frac{Ay_k+B}{A}\le \frac{z_k}A.
\end{align}
This completes the proof of the lemma.
\end{proof}

\begin{lemma}\label{lemI}
Let $A,B,C,D$ be positive numbers satisfying $D<A<1$ and $\{x_k\}_{k\in\NN}$
be a sequence of positive numbers satisfying the inequality
\begin{align}
x_{k+1} \le \big((1-A)^2x_k^2+B^2\big)^{1/2}+C+Dx_k.
\end{align}
Then, for every $k\ge 0$, we have
\begin{align}
x_k\le (1-A+D)^k x_0 + \frac{C}{A-D}+\frac{B}{\sqrt{(A-D)(2-A-D)}}.
\end{align}
\end{lemma}

\begin{proof}
We start by setting
$$
E = \frac{B}{\sqrt{(A-D)(2-A-D)}},\qquad F = C+ (A-D)E
$$
and by defining a new sequence $\{y_k\}_{k\in\NN}$ by
$y_0 = x_0+ E$ and
$$
y_{k+1} = (1-A+D)y_k + F.
$$
Our goal is to prove that $y_k\ge x_k\vee E$ for every $k$.
This claim is clearly true for $k=0$. Let us assume that it
is true for the value $k$ and prove its validity for $k+1$.
Since the function $x\mapsto \sqrt{x^2+a^2} - x$ is decreasing,
we have
\begin{align}
x_{k+1}
		&\le \sqrt{(1-A)^2y_k^2+B^2}+C+Dy_k\\
		&\le (1-A+D)y_k +C + \sqrt{(1-A)^2y_k^2+B^2}-(1-A)y_k\\
		&\le (1-A+D)y_k +C + \sqrt{(1-A)^2E^2+B^2}-(1-A)E = y_{k+1}.
\end{align}
On the other hand,
\begin{align}
y_{k+1}
		&\ge (1-A+D)y_k + (A-D)E\\
		&\ge (1-A+D)E + (A-D)E = E.
\end{align}
This implies, in particular, that $x_k\le y_k$ for every $k\in\NN$.
Since $\{y_k\}$ satisfies a first-order linear recursion, we get
$y_k = (1-A+D)^ky_0 + F(1-(1-A+D)^k)/(A-D)$. 
\end{proof}
{\renewcommand{\addtocontents}[2]{}

\section*{Acknowledgments}
The work of AD was partially supported by the grant
Investissements d'Avenir (ANR-11-IDEX-0003/Labex Ecodec/ANR-11-LABX-0047).
}%

{\renewcommand{\addtocontents}[2]{}
\bibliography{Literature1}}

\end{document}